\documentclass[11pt]{amsproc}

\usepackage{enumerate}
\usepackage{csquotes}

\usepackage{xargs}
\usepackage{amsmath,amssymb,amsfonts,amsthm}
\usepackage{mathtools}

\usepackage[backrefs]{amsrefs}

\newcommandx{\set}[2][2=\empty]{\{#1\ifx#2\empty\else\,|\,#2\fi\}}
\newcommandx{\lrset}[2][2=\empty]{\left\{#1\ifx#2\empty\else\,\middle|\,#2\fi\right\}}
\newcommand{\card}[1]{\lvert#1\rvert}
\newcommandx{\gensubgrp}[2][2=\empty]{\langle#1\ifx#2\empty\else\,|\,#2\fi\rangle}
\newcommandx{\gensubsp}[2][2=\empty]{\langle#1\ifx#2\empty\else\,|\,#2\fi\rangle}
\newcommand{\nats}{\mathbb{N}}
\newcommand{\ints}{\mathbb{Z}}
\newcommand{\finfield}{\mathbb{F}}
\newcommand{\abs}[1]{\lvert#1\rvert}
\DeclareMathOperator{\ord}{ord}
\DeclareMathOperator{\tr}{tr}
\DeclareMathOperator{\N}{N}
\DeclareMathOperator{\id}{id}
\DeclareMathOperator{\diag}{diag}
\newcommand{\rest}[1]{#1\rvert}
\DeclareMathOperator{\rk}{rk}
\DeclareMathOperator{\supp}{supp}
\DeclareMathOperator{\lcm}{lcm}
\newcommand{\floor}[1]{\lfloor#1\rfloor}
\DeclareMathOperator{\codim}{codim}
\DeclareMathOperator{\fix}{fix}
\newcommand{\powset}{\mathcal P}
\DeclareMathOperator{\Sub}{Sub}
\newcommand{\trivgrp}{\mathbf{1}}
\DeclareMathOperator{\Cycgrp}{C}
\DeclareMathOperator{\Sym}{Sym}
\DeclareMathOperator{\Sgrp}{S}
\newcommand{\M}{\mathbf{M}}
\DeclareMathOperator{\GL}{GL}
\DeclareMathOperator{\Sp}{Sp}
\DeclareMathOperator{\GO}{GO}
\DeclareMathOperator{\GU}{GU}
\newcommand{\PM}{\mathrm{P}\mathbf{M}}
\DeclareMathOperator{\PGL}{PGL}
\DeclareMathOperator{\PSp}{PSp}
\DeclareMathOperator{\PGO}{PGO}
\DeclareMathOperator{\PGU}{PGU}
\DeclareMathOperator{\stab}{stab}
\newcommand{\C}{\mathbf{C}}
\newcommand{\Z}{\mathbf{Z}}
\DeclareMathOperator{\Aut}{Aut}
\DeclareMathOperator{\Gal}{Gal}
\DeclareMathOperator{\orb}{orb}
\newcommand{\calH}{\mathcal{H}}
\newcommand{\calS}{\mathcal{S}}
\newcommand{\calU}{\mathcal{U}}
\newcommand{\calV}{\mathcal{V}}

\theoremstyle{plain}
\newtheorem{theorem}{Theorem}
\newtheorem{lemma}{Lemma}
\newtheorem{corollary}{Corollary}

\theoremstyle{definition}
\newtheorem{remark}{Remark}

\title[Isomorphism questions for metric ultraproducts]{Isomorphism questions for metric ultraproducts of finite quasisimple groups}
\author{Jakob Schneider}
\address{J.~Schneider, TU Dresden, 01062 Dresden, Germany}
\email{jakob.schneider@tu-dresden.de}

\begin{document}
\begin{abstract}
	New results on metric ultraproducts of finite simple groups are established. We show that the isomorphism type of a simple metric ultraproduct of groups $X_{n_i}(q)$ ($i\in I$) for $X\in\set{\PGL,\PSp,\PGO^{(\varepsilon)},\PGU}$ ($\varepsilon=\pm$) along an ultrafilter $\calU$ on the index set $I$ for which $n_i\to_\calU\infty$ determines the type $X$ and the field size $q$ up to the possible isomorphism of a metric ultraproduct of groups $\PSp_{n_i}(q)$ and a metric ultraproduct of groups $\PGO_{n_i}^{(\varepsilon)}(q)$. This extends results of Thom and Wilson~\cite{thomwilson2014metric}.
\end{abstract}

\maketitle

\section{Introduction}
In \cite{thomwilson2014metric,thomwilson2018some} Thom and Wilson discussed various properties of metric ultraproducts of finite simple groups. In particular, they asked the question which such ultraproducts can be isomorphic. In Theorem~2.2 of \cite{thomwilson2014metric}, a metric ultraproduct of alternating groups is distinguished from a metric ultraproduct of classical groups of Lie type, where the permutation degrees resp.\ dimensions of the natural module tend to infinity. This is done by considering the structure of centralizers of torsion elements in these groups (see Theorems~2.8 and~2.9 of \cite{thomwilson2014metric}).
In the case of a metric ultraproduct of classical groups of Lie type, in Theorem~2.8 of \cite{thomwilson2014metric}, investigating the structure of such centralizers of semisimple and unipotent torsion elements, Thom and Wilson even extract the \enquote*{limit characteristic} of the group. At the end of Section~2 of \cite{thomwilson2014metric} they ask which metric ultraproducts of classical groups of different types can be isomorphic. 

In this note, we will give an almost complete answer to this question in the case when the field sizes are bounded. We will show that for a metric ultraproduct of alternating or classical groups of Lie type of unbounded rank over fields of bounded size one can extract the Lie type (up to one exception). Also one can extract the \enquote*{limit field size}. Our results are summed up in Theorem~\ref{thm:non_iso_main} below. To state it, we first need to introduce some notation.

\vspace{3mm}

Let $\calH=(H_i)_{i\in I}$ be a sequence of groups where either $H_i=\Sgrp_{n_i}$ is a symmetric group or $H_i=X_i(q_i)$ is a classical group of Lie type $X_i$ over the finite field $\finfield_{q_i}$ with $q_i$ elements (resp.\ $\finfield_{q_i^2}$ in the unitary case; $i\in I$). In the latter case, we let each $X_i$ be one of $\GL_{n_i}$, $\Sp_{2m_i}$, $\GO_{2m_i}^{\pm}$, $\GO_{2m_i+1}$ ($q_i$ odd), or $\GU_{n_i}$ for suitable $m_i,n_i\in\ints_+$ ($i\in I$).

Recall that a \emph{norm} $\ell$ on a group $H$ is a function $H\to[0,\infty]$ such that $\ell(h)=0$ iff $h=1_H$, $\ell(h)=\ell(h^{-1})=\ell(h^g)$, and $\ell(gh)\leq\ell(g)+\ell(h)$ for all $g,h\in H$. Call a pair $(H,\ell)$, where $H$ is a group and $\ell$ a norm on it, a \emph{normed group}. Recall that the metric ultraproduct of a sequence of normed groups $(H_i,\ell_i)_{i\in I}$ along an ultrafilter $\calU$ on the set $I$ is defined as the quotient $\left.\prod_{i\in I}{H_i}\middle/N\right.$, where $N\coloneqq\set{(h_i)_{i\in I}\in\prod_{i\in I}{H_i}}[\lim_\calU{\ell_i(h_i)}=0]$ is a normal subgroup.

Throughout, let $G\coloneqq\calH_\calU^{\rm met}$ be the metric ultraproduct of the groups $H_i$ from above equipped with the normalized Hamming norm $\ell_{\rm H}(\sigma)\coloneqq\card{\supp(\sigma)}/n=\card{\set{x\in\set{1,\ldots,n}}[x.\sigma\neq x]}/n$ resp.\ the normalized rank norm $\ell_{\rk}(g)\coloneqq\rk(1-g)/n$ when $H_i$ is a symmetric resp.\ a classical linear group of Lie type. Assume that the permutation degrees resp.\ dimensions of the natural module $n_i$ of $H_i$ ($i\in I$) tend to infinity along $\calU$.

Note that, since $\calU$ is an ultrafilter, we may assume that each group $H_i$ is of the same type, i.e., all groups $H_i$ are either symmetric, linear, symplectic, orthogonal, or unitary groups. In these five distinct cases, we write $\Sgrp_\calU$, $\GL_\calU$, $\Sp_\calU$, $\GO_\calU$, or $\GU_\calU$ for $G$. Also, when the field sizes $q_i$ are bounded, we may assume that $q_i=q$ is constant ($i\in I$), setting $q\coloneqq\lim_\calU{q_i}$. Throughout, set $Z\coloneqq\Z(G)$ to be the \emph{center} of $G$ and $\overline{G}\coloneqq G/Z$.

If the groups $H_i$ ($i\in I$) are symmetric groups, then $Z=\trivgrp$ and $\overline{G}=G$.
Now assume that all groups $H_i$ are of type $X(q_i)$ ($i\in I$; i.e., they are not symmetric groups). Then $\overline{G}=G/Z=\overline{\calH}_\calU^{\rm met}$ is the metric ultraproduct of the groups $\overline{H}_i\coloneqq H_i/\Z(H_i)$ with respect to the projective rank norm $\ell_{\rm pr}(\overline{h})\coloneqq\inf\set{\ell_{\rk}(h)}[h\text{ is a lift of }\overline{h}\text{ in }H]$, which is defined on the general projective linear group. By the results from \cite{liebeckshalev2001diameters}, $\overline{G}$ is the unique simple quotient of $G$.

Similarly to the above, write $\PGL_\calU$, $\PSp_\calU$, $\PGO_\calU$, or $\PGU_\calU$ for $\overline{G}$ when all the groups $H_i$ ($i\in I$) are linear, symplectic, orthogonal, or unitary groups. Moreover, in this case, if all the fields $\finfield_{q_i}$ ($i\in I$) are equal to $\finfield_q$ (or $\finfield_{q^2}$ in the unitary case), we write $\GL_\calU(q)$, $\Sp_\calU(q)$, $\GO_\calU(q)$, $\GU_\calU(q)$ resp.\ $\PGL_\calU(q)$, $\PSp_\calU(q)$, $\PGO_\calU(q)$, $\PGU_\calU(q)$ for $G$ resp.\ $\overline{G}$. Write $\M_n(k)$ for the matrix ring of degree $n$ over the field $k$ and $\PM_n(k)$ for the associated projective space $(\M_n(k)\setminus\set{0})/k^\times$. Set $\M_n(q)\coloneqq\M_n(\finfield_q)$ and $\PM_n(q)\coloneqq\PM_n(\finfield_q)$. Also write $\M_\calU$, $\M_\calU(q)$ resp.\ $\PM_\calU$, $\PM_\calU(q)$ for the metric ultraproduct of the spaces $\M_{n_i}(q_i)$, $\M_{n_i}(q)$ resp.\ $\PM_{n_i}(q_i)$, $\PM_{n_i}(q)$ with respect to the metrics $d_{\rk}(g,h)\coloneqq\rk(g-h)/n$ and $d_{\rm pr}(\overline{g},\overline{h})\coloneqq\inf\set{d_{\rk}(g,h)}[g,h\text{ are lifts of }\overline{g},\overline{h}]$ ($i\in I$), so that $\GL_\calU\subseteq\M_\calU$, $\GL_\calU(q)\subseteq\M_\calU(q)$, $\PGL_\calU\subseteq\PM_\calU$, $\PGL_\calU(q)\subseteq\PM_\calU(q)$. Throughout, if not stated otherwise, let $k=\finfield_q$ when $G$ is classical non-unitary and $k=\finfield_{q^2}$ in the unitary case.

If all $H_i$ ($i\in I$) are symplectic, orthogonal, or unitary, write $f_i$ for the sesquilinear forms stabilized (and $Q_i$ for the quadratic form in the orthogonal case in characteristic two). 

\vspace{3mm}

The main result of this article is now as follows.

\begin{theorem}\label{thm:non_iso_main}
	Let $\overline{G}\cong\overline{G}_1\cong\overline{G}_2$ with $G_j=X_{j\calU_j}(q_j)$, where $X_j\in\set{\GL,\Sp,\GO,\GU}$ ($j=1,2$). Then it holds that $q_1=q_2$. Also we must have $X_1=X_2$ or  $\set{X_1,X_2}=\set{\Sp,\GO}$. Moreover, an ultraproduct $\overline{X}_{1\calU_1}$ where the sizes $q_i$ of the finite fields $\finfield_{q_i}$ ($i\in I_1$) tend to infinity along $\calU_1$ cannot be isomorphic to an ultraproduct $\overline{X}_{2\calU_2}(q)$.
\end{theorem}

Let us conclude this introduction by saying some words about the proof of Theorem~\ref{thm:non_iso_main}. Our strategy is to compute double centralizers of semisimple torsion elements of a fixed order $o\in\ints_+$ in the above metric ultraproducts. If the sizes $q_i$ ($i\in I$) of the fields $\finfield_{q_i}$ are bounded, it turns out that these are always finite abelian groups. Then we consider the maximal possible exponent which such a double centralizer can have. It turns out that this data is enough to determine the limit field size $q$ and the Lie type (up to the exception mentioned in Theorem~\ref{thm:non_iso_main}). If the field sizes $q_i$ ($i\in I$) tend to infinity, a double centralizer of such a torsion element of order $o>2$ is always infinite.  Throughout the article we will make frequent use of the classification given in Subsection~3.4.2~\S1 of \cite{schneider2019phd}.

\section{Description of conjugacy classes in $\Sgrp_\calU$, $\GL_\calU(q)$, and $\PGL_\calU(q)$}

In this section, we give a description of the conjugacy classes of groups of type $\Sgrp_\calU$ or $\PGL_\calU(q)$. We will make use of this in the subsequent sections. 

We start with some definitions. Write $\Sgrp_n=\Sym(\underline{n})$ for the symmetric group acting on the set $\underline{n}\coloneqq\set{1,\ldots,n}$. For a permutation $\sigma\in\Sgrp_n$ and $k\in\ints_+$, let $C_k(\sigma)\subseteq\Sgrp_n$ denote the set of all $k$-cycles of $\sigma$ and $c_k(\sigma)\coloneqq\card{C_k(\sigma)}$ denote the number of $k$-cycles of $\sigma$. Moreover, let $\Omega_k(\sigma)$ be the set on which all $k$-cycles of $\sigma$ act. Set $n_k(\sigma)\coloneqq\card{\Omega_k(\sigma)}=kc_k(\sigma)$ to be the cardinality of $\Omega_k(\sigma)$. Call the permutation $\sigma\in\Sgrp_n$ \emph{$k$-isotypic} when $n_k(\sigma)=n$, i.e., $\sigma$ has only cycles of length $k$. Call $\sigma$ \emph{isotypic} if it is $k$-isotypic for one number $k\in\ints_+$. Similarly, for $g\in\M_n(k)$ and a monic primary polynomial $\chi\in k[X]$ (i.e., a power of an irreducible polynomial) let $c_\chi(g)$ be the number of Frobenius blocks $F(\chi)$ in the generalized Jordan normal form 
$$
g\cong\bigoplus_{\substack{\chi\in k[X]\\ \chi\text{ monic primary}}}{F(\chi)^{\oplus c_\chi(g)}}
$$
of $g$. Let $V_\chi(g)$ be the (not unique) subspace on which $g$ acts as $F(\chi)^{\oplus c_\chi(g)}$ with respect to the above normal form and set $n_\chi(g)\coloneqq\dim(V_\chi(g))$. Then, of course, $n_\chi(g)=\deg(\chi)c_\chi(g)$. For a monic polynomial $\chi\in k[X]$, say that $g$ is \emph{$\chi$-isotypic} if $g\cong F(\chi)^{\oplus c}$ for some $c\in\ints_+$. 

At first we consider $\GL_\calU(q)$ instead of $\PGL_\calU(q)$. Throughout this article, all polynomials from $k[X]$ that occur are meant to be \emph{monic} polynomials. 

\vspace{3mm}

\emph{Conjugacy classes in $\Sgrp_\calU$ and $\GL_\calU(q)$.}
For an integer $k\in\ints_+$ and a polynomial $\xi\in k[X]$ define $r_k(\sigma)\coloneqq\card{\fix(\sigma^k)}/n\text{ resp.\ }r_\xi(g)\coloneqq\dim(\ker(\xi(g)))/n$ for $\sigma\in\Sgrp_n$ resp.\ $g\in\M_n(q)$. Here $\fix(\sigma)\coloneqq\set{x\in\underline{n}}[x.\sigma=x]$ is the set of \emph{fixed points} of the permutation $\sigma$. Extend this definition to $\Sgrp_\calU$ and $\M_\calU(q)$ by setting $r_k(\sigma)\coloneqq\lim_\calU{r_k(\sigma_i)}$ and $r_\xi(g)\coloneqq\lim_\calU{r_\xi(g_i)}$ for $\sigma=\overline{(\sigma_i)}_{i\in I}$ resp.\ $g=\overline{(g_i)}_{i\in I}$.
Both expressions are well-defined, since for $\sigma=\overline{(\sigma_i)}_{i\in I}=\overline{(\tau_i)}_{i\in I}\in\Sgrp_\calU$ one has
\begin{align}
\frac{1}{n_i}\left|\card{\fix(\sigma_i^k)}-\card{\fix(\tau_i^k)}\right|&\leq d_{\rm H}(\sigma_i^k,\tau_i^k)\nonumber\\
&\leq d_{\rm H}(\sigma_i^k,\sigma_i^{k-1}\tau_i)+\cdots+d_{\rm H}(\sigma_i\tau_i^{k-1},\tau_i^k)\label{eq:wll_def_r_k}\\
&=kd_{\rm H}(\sigma_i,\tau_i)\to_\calU 0.\nonumber
\end{align}
Similarly, if $g=\overline{(g_i)}_{i\in I}=\overline{(h_i)}_{i\in I}\in\GL_\calU(q)$ and $\xi=a_0+a_1X+\cdots+a_{k-1}X^{k-1}+X^k\in k[X]$ we have
\begin{align}
\frac{1}{n_i}\left|\dim(\ker\xi(g_i))-\dim(\ker\xi(h_i))\right|&\leq d_{\rm rk}(\xi(g_i),\xi(h_i))\nonumber\\
&\leq\sum_{j=0}^k{d_{\rm rk}(a_jg_i^j,a_jh_i^j)}\nonumber\\
&\leq\sum_{j=0}^k{d_{\rm rk}(g_i^j,h_i^j)}\label{eq:wll_def_r_chi}\\
&\leq\left(\sum_{j=0}^k{j}\right)d_{\rm rk}(g_i,h_i)=\binom{k+1}{2}d_{\rm rk}(g_i,h_i)\nonumber
\end{align}
and the latter tends to zero along $\calU$. Here we used the same trick as in Estimate~\eqref{eq:wll_def_r_k} above to bound $d_{\rm rk}(g_i^j,h_i^j)$ by $jd_{\rm rk}(g_i,h_i)$ ($j=0,\ldots,k$) in Estimate~\eqref{eq:wll_def_r_chi}. Write $r(\sigma)\coloneqq(r_k(\sigma))_{k\in\ints_+}$ and $r(g)\coloneqq (r_\xi(g))_{\xi\in k[X]}$.
Now define $q_k(\sigma)$ for $k\in\ints_+$ via the equality
$$
\sum_{d\mid k}{q_d(\sigma)}=r_k(\sigma),
$$ 
for all $k\in\ints_+$. Write $q(\sigma)\coloneqq(q_k(\sigma))_{k\in\ints_+}$. Applying M\"obius inversion, we obtain
$$
q_k(\sigma)=\sum_{d\mid k}{\mu(k/d)r_d(\sigma)}.
$$ 
Alternatively, one can think of $q_k(\sigma)$ as the $\calU$-limit of the normalized cardinality of the support of all $k$-cycles in $\sigma_i$ ($i\in I$), i.e., 
$$
q_k(\sigma)=\lim_\calU{n_k(\sigma_i)/n_i}=k\lim_\calU{c_k(\sigma_i)/n_i}.
$$

Similarly, for $\chi\in k[X]$ primary define $q_\chi(g)$ via the equality
$$
r_\xi(g)=\sum_{\chi\text{ primary}}{\frac{\deg(\gcd\set{\chi,\xi})}{\deg(\chi)} q_\chi(g)},
$$
for all polynomials $\xi\in k[X]$. Here $\gcd S$ is the greatest common divisor of the elements from $S$. Write $q(g)\coloneqq(q_\chi(g))_{\chi\text{ primary}}$. Alternatively, one can think of $q_\chi(g)$ as the $\calU$-limit of the normalized dimensions of the (not unique) subspaces $V_\chi(g_i)$ ($i\in I$), i.e.,
$$
q_\chi(g)=\lim_\calU{n_\chi(g_i)/n_i}=k_\chi\lim_\calU{c_\chi(g_i)/n_i},
$$
where $k_\chi=\deg(\chi)=e\deg(i)$. This is because, when $g$ acts as $F(\chi)$ and $\xi\in k[X]$, then $\dim(\ker(\xi(g)))=\deg(\gcd\set{\chi,\xi})$.

We claim that the conjugacy classes in $\Sgrp_\calU$ resp.\ $\M_\calU(q)$ are in one-to-one correspondence with all tuples $(q_k(\sigma))_{k\in\ints_+}$ resp.\ $(q_\chi(g))_{\chi \text{ primary}}$, where the only condition on the sequences are that 
$$
\sum_{k\in\ints_+}{q_k(\sigma)}\leq 1\text{ resp.\ }\sum_{\chi\text{ primary}}{q_\chi(g)}\leq 1.
$$
Here we let $\GL_\calU(q)$ act on $\M_\calU(q)$ by conjugation.
The element $g$ lies in $\GL_\calU(q)$ if and only if $q_\chi(g)=0$ for $\chi=X^e$ ($e\geq 1$).
Indeed, one sees easily that $r(\sigma)$ resp.\ $r(g)$ is conjugacy invariant, and so is $q(\sigma)$ resp.\ $q(g)$ for $\sigma\in\Sgrp_\calU$ and $g\in\M_\calU(q)$.

To see the converse, let $\sigma=\overline{(\sigma_i)}_{i\in I},\tau=\overline{(\tau_i)}_{i\in I}\in\Sgrp_\calU$ resp.\ $g=\overline{(g_i)}_{i\in I},h=\overline{(h_i)}_{i\in I}\in\M_\calU(q)$ be elements such that $q(\sigma)=q(\tau)$ resp.\ $q(g)=q(h)$. 

Find a sequence $(N_i)_{i\in I}$ tending to infinity along $\calU$ such that 
$$
\sum_{k=1}^{N_i}{\abs{q_k(\sigma)-q_k(\sigma_i)}},\ \sum_{k=1}^{N_i}{\abs{q_k(\tau)-q_k(\tau_i)}}\to_\calU 0
$$
resp.\ such that
$$
\sum_{\substack{\chi\text{ primary}\\ \deg(\chi)\leq N_i}}{\abs{q_\chi(g)-q_\chi(g_i)}},\sum_{\substack{\chi\text{ primary}\\ \deg(\chi)\leq N_i}}{\abs{q_\chi(h)-q_\chi(h_i)}}\to_\calU 0.
$$
Then by the triangle inequality
$$
\frac{1}{n_i}\sum_{k=1}^{N_i}{\abs{n_k(\sigma_i)-n_k(\tau_i)}}\to_\calU 0
$$
resp.\ 
$$
\frac{1}{n_i}\sum_{\substack{\chi\text{ primary}\\ \deg(\chi)\leq N_i}}{\abs{\dim(V_\chi(g_i))-\dim(V_\chi(h_i))}}\to_\calU 0.
$$

Hence we can conjugate a big part of $\bigsqcup_{k=1}^{N_i}{\Omega_k(\sigma_i)}$ equivariantly  to a big part of $\bigsqcup_{k=1}^{N_i}{\Omega_k(\tau_i)}$ resp.\ an almost fulldimensional part of 
$$
\bigoplus_{\substack{\chi\text{ primary}\\ \deg(\chi)\leq N_i}}{V_\chi(g_i)}\quad\text{ to }\bigoplus_{\substack{\chi\text{ primary}\\ \deg(\chi)\leq N_i}}{V_\chi(h_i)}$$ 
equivariantly (with no error in the limit; here again $V_\chi(g_i)$ resp.\ $V_\chi(h_i)$ are not unique). The remaining part of $\sigma_i$ and $\tau_i$ resp.\ $g_i$ and $h_i$ can be modified into one big cycle resp.\ one big Frobenius block with no change of $\sigma$ and $\tau$ resp.\ $g$ and $h$, since $N_i\to_\calU\infty$. Then we conjugate this cycle resp.\ Frobenius block onto the other.

\vspace{3mm}

\emph{The case of $\PM_\calU(q)$.}
Let the group $k^\times=\finfield_q^\times$ act on all (monic) polynomials $\xi\in k[X]$ by $\xi.z\coloneqq z^{-k_\xi}\xi(z X)$, where $k_\xi\coloneqq\deg(\xi)$. Extend this action to all tuples $q=(q_\chi)_{\chi\text{ primary}}$ with $\sum_{\chi\text{ primary}}{q_\chi}\leq 1$ via 
$$
q.z=(q_{\chi.z})_{\chi\text{ primary}}
$$ 
and denote by $\overline{q}$ the orbit $\orb_{k^\times}(q)$ of $q$ under this action of $k^\times$.

Let $\overline{G}=\PGL_\calU(q)$. We claim that the conjugacy classes of elements $\overline{g}\in\PM_\calU(q)$ are classified by the bijection $\overline{g}^{\overline{G}}\mapsto\overline{(q_\chi(g))}_{\chi\text{ primary}}$, where $g$ is any lift of $\overline{g}$ in $\M_\calU(q)$ (here we exclude the tuple $q$ where $q_X=1$ and $q_\chi=0$ otherwise).

Indeed, the map is well-defined, since any other $h$ such that $\overline{h}=\overline{g}\in\PM_\calU(q)$ is of the form $z g$ for some $z\in k^\times$ (as $k=\finfield_q$ is finite), so that $(q_\chi(h))_{\chi\text{ primary}}=(q_\chi(z g))_{\chi\text{ primary}}=(q_{\chi.z}(g))_{\chi\text{ primary}}=q(g).z$. Also $q$ is constant on conjugacy classes of $\M_\calU(q)$ (under the action of $\GL_\calU(q)$).

Conversely, if $q_\chi(h)=q_{\chi.z}(g)=q_\chi(z g)$ for some fixed $z\in k^\times$ and all $\chi\in k[X]$ primary, then from the above we derive that the elements $g$ and $z^{-1}h$ of $\M_\calU(q)$ are conjugate, so that $\overline{g}$ and $\overline{h}$ are conjugate in $\PM_\calU(q)$. This proves the claim.

\begin{remark}\label{rmk:oth_grps}
	For $G$ of type $\Sp_\calU(q)$, $\GO_\calU(q)$, or $\GU_\calU(q)$ the conjugacy classes of elements $g\in G$ for which $\sum_{\chi\text{ primary}}{q_\chi(g)}=1$ are still characterized by the tuples $(q_\chi(g))_{\chi\text{ primary}}$. The only necessary additional restriction on these tuples is that $q_\chi(g)=q_{\chi^\ast}(g)$, where $\chi^\ast=a_0^{-\sigma}X^{\deg(\chi)}\chi^\sigma(X^{-1})$ is the \emph{dual} polynomial of $\chi=a_0+a_1X+\cdots+a_{k-1}X^{k-1}+X^k$ (here $\sigma\colon k\to k$ is the identity if $G$ is not of type $\GU_\calU(q)$ and the $q$-Frobenius $\finfield_{q^2}\to\finfield_{q^2}$; $x\mapsto x^q$ if $G$ is of type $\GU_\calU(q)$; $\chi^\sigma$ is the polynomial $\chi$ with coefficients twisted by $\sigma$; cf.\ Subsection~3.4.2~\S1 of \cite{schneider2019phd}).
	
	Indeed, assume $g=\overline{(g_i)}_{i\in I},h=\overline{(h_i)}_{i\in I}\in G$, $q(g)=q(h)$, $q_\chi(g)=q_{\chi^\ast}(g)$ for all $\chi\in k[X]$ primary, and 
	$$
	\sum_{\chi\text{ primary}}{q_\chi(g)}=\sum_{\chi\text{ primary}}{q_\chi(h)}=1.
	$$ 
	Then $g$ is conjugate to $h$.
	
	This holds, since on all but constantly many Frobenius blocks $F(\chi)$ of $g_i$ resp.\ $h_i$ ($i\in I$), where $\chi$ is a self-dual primary polynomial, and on all but constantly many blocks $F(\chi)\oplus F(\chi^\ast)\cong F(\chi\chi^\ast)$, where $\chi$ is not self-dual and primary, according to Subsection~3.4.2~\S1 of \cite{schneider2019phd}, the form $f_i$ (and in characteristic two the quadratic form $Q_i$) is uniquely determined up to linear equivalence, so that we can map these blocks of $g_i$ to such blocks of $h_i$ and extend this partial map by Witt's lemma.
	
	Conversely, if we have a tuple $(q_\chi)_{\chi\text{ primary}}$ such that $\sum_{\chi\text{ primary}}{q_\chi}=1$ and $q_\chi=q_{\chi^\ast}$, we can see from Fact~3.40 of \cite{schneider2019phd} that there exists an element $g\in G$ such $q_\chi(g)=q_\chi$ for all $\chi\in k[X]$ primary.
	
	Recall that $\overline{G}=G/Z$, where $Z=\Z(G)$. For a tuple $q=(q_\chi)_{\chi\text{ primary}}$ let $\overline{q}$ denote its orbit $\orb_Z(q)$ under the action of $Z\leq k^\times$. Then for the elements $\overline{g}\in\overline{G}$ (i.e., $\overline{G}$ is one of $\PSp_\calU(q)$, $\PGO_\calU(q)$, or $\PGU_\calU(q)$) such that $\sum_{\chi\text{ primary}}{q_\chi(g)}=1$ for one lift $g\in G$ of $\overline{g}$, the same characterization as for $\PGL_\calU(q)$ above holds by the same argument. Again we need to restrict the tuples $q=(q_\chi)_{\chi\text{ primary}}$ so that $q_\chi=q_{\chi^\ast}$ for all $\chi$ primary.
	
	However, we conjecture that the above characterization for $G$ of type $\Sp_\calU(q)$, $\GO_\calU(q)$ or $\GU_\calU(q)$ is false if 
	$$
	\sum_{\chi\text{ primary}}{q_\chi(g)}<1
	$$
	for an element $g\in G$.
\end{remark}

\begin{remark}\label{rmk:comm_*_Z}
	It is easy to see that $(\chi.z)^\ast=\chi^\ast.z$ for $z\in Z$. Indeed, since $z^\sigma=z^{-1}$ for $z\in Z$, we have 
	\begin{align*}
	(\chi.z)^\ast&=(z^{-k}\chi(zX))^\ast=(z^k)^\sigma a_0^{-\sigma} X^k(z^{-k})^\sigma\chi^\sigma(z^\sigma X^{-1})\\
	&=a_0^{-\sigma}X^k\chi^\sigma((zX)^{-1})=z^{-k}\chi^\ast(zX)=\chi^\ast.z,
	\end{align*}
	where $\chi=a_0+a_1X+\cdots+a_{k-1}X^{k-1}+X^k$ and $\sigma\in\Aut(k)$ is defined as above in Remark~\ref{rmk:oth_grps}.
\end{remark}

\section{Characterization of torsion elements in $\Sgrp_\calU$, $\GL_\calU(q)$, and $\PGL_\calU(q)$}

In this section, we characterize torsion elements in metric ultraproducts of the above type.
At first note that an invertible element in $\M_\calU(q)$, i.e., an element of $\GL_\calU(q)$, is algebraic over $k=\finfield_q$ if and only if it is torsion. Indeed, if $g$ is torsion, then $g^o-1=0$ for some integer $o\geq 1$. Conversely, if $g$ is algebraic and invertible, let $\chi\in k[X]$ be the its minimal polynomial. Setting $o\coloneqq\card{(k[X]/(\chi))^\times}<\infty$ one sees that $g^o=1$ as $g$ is invertible.

Here comes the promised characterization of torsion elements.

\begin{lemma}\label{lem:char_tor_el}
	An element $\sigma\in\Sgrp_\calU$ resp.\ $g\in\GL_\calU(q)$ is torsion if and only if there is $N\in\ints_+$ such that 
	$$
	\sum_{k=1}^N{q_k(\sigma)}=1\text{ resp.\ }\sum_{\substack{\chi\text{ primary}\\ \deg(\chi)\leq N}}{q_\chi(g)}=1.
	$$	
	An element $\overline{g}\in\PGL_\calU(q)$ is torsion if and only if any lift $g\in\GL_\calU(q)$ is torsion.
\end{lemma}

\begin{proof}
	Write $\lcm S$ for the least common multiple of the elements of $S$. Indeed, if the above two conditions are fulfilled, then writing $o\coloneqq\lcm\set{1,\ldots,N}$ resp.\ $o\coloneqq\lcm\set{\card{(k[X]/(\chi))^\times}}[\chi\text{ primary, }\deg(\chi)\leq N]$, we have $\ell_{\rm H}(\sigma_i^o)\to_\calU 0$ resp.\ $\ell_{\rm rk}(g_i^o)\to_\calU 0$ meaning that $\sigma^o=1$ resp.\ $g^o=1$. Conversely, if we assume $\sigma^o=1$ resp.\ $g^o=1$, we get that $\ell_{\rm H}(\sigma_i^o)\to_\calU 0$ resp.\ $\ell_{\rm rk}(g_i^o)\to_\calU 0$, meaning that, asymptotically, the $d$-cycles in $\sigma_i$ for $d\mid o$ support the whole set resp.\ all Frobenius blocks $F(\chi)$ for $\chi\mid X^o-1$ primary support the whole vector space, so taking $N\coloneqq o$ above, we get the converse direction.
	
	The last statement follows, since the kernel of the surjective homomorphism $\GL_\calU(q)\to\PGL_\calU(q)$ equals $k^\times=\finfield_q^\times$, which is finite. Hence, if $g\in\GL_\calU(q)$ represents $\overline{g}\in\PGL_\calU(q)$ and the latter is of order $o<\infty$, we have that $\ord(g)\mid o(q-1)<\infty$.
\end{proof}

\section[Faithful action of $\Sgrp_\calU$ and $\PGL_\calU(q)$]{Faithful action of $\Sgrp_\calU$ and $\PGL_\calU(q)$ on the Loeb space and the associated continuous geometry}\label{sec:ffl_act_met_ups_Ln_sp}

In this section, we show that the groups $\Sgrp_\calU$ and $\PGL_\calU(q)$ faithfully act on natural associated objects.
For this purpose we need the so-called \emph{Loeb space}
$$
L(n_i)_{i\in I}\coloneqq\left(\calS,\mu\right)
$$
resp.\ its vector space analog, the \emph{continuous geometry}
$$
V(n_i)_{i\in I}\coloneqq\left(\calV,\dim\right),
$$
which are associated naturally to the metric ultraproduct $\Sgrp_\calU$ resp.\ $\PGL_\calU(q)$.

Here $\calS$ resp.\ $\calV$ equals $\prod_{i\in I}{\powset(\underline{n_i})}$ resp.\ $\prod_{i\in I}{\Sub(k^{n_i})}$ modulo the equivalence relation $(S_i)_{i\in I}\sim(T_i)_{i\in I}$ resp.\ $(U_i)_{i\in I}\sim(V_i)_{i\in I}$ iff $\mu_i(S_i\triangle T_i)\to_\calU 0$ resp.\ $\dim_i(U_i+V_i)-\dim_i(U_i\cap V_i)\to_\calU 0$, where $\mu_i$ resp.\ $\dim_i$ is the normalized counting measure resp.\ dimension on $\underline{n_i}$ resp.\ $k^{n_i}$ (and $A\triangle B$ denotes the symmetric difference of the sets $A$ and $B$).
Then one defines $\mu$ resp.\ $\dim$ by 
$$
\mu(S)=\mu(\overline{(S_i)}_{i\in I})\coloneqq\lim_\calU{\mu_i(S_i)}
$$ 
and 
$$
\dim(V)=\dim(\overline{(V_i)}_{i\in I})\coloneqq\lim_\calU{\dim_i(V_i)}.
$$
It is easy to check that both are well-defined in this way. Also the operations $\cup,\cap$ resp.\ $+,\cap$ are inherited to $\calS$ resp.\ $\calV$ in a natural way, e.g., $\overline{(S_i)}_{i\in I}\cap\overline{(T_i)}_{i\in I}\coloneqq\overline{(S_i\cap T_i)}_{i\in I}$. Write $S\subseteq T$ resp.\ $U\leq V$ iff $\mu(S\cap T)=\mu(S)$ resp.\ $\dim(U\cap V)=\dim(U)$. Call a permutation of $\calS$ resp.\ $\calV$ an automorphism iff it preserves $\mu$ resp.\ $\dim$ and the relation $\subseteq$ resp.\ $\leq$.

Then one observes that $\Sgrp_\calU$ resp.\ $\PGL_\calU(q)$ is faithfully represented as group of automorphisms of $(\calS,\mu)$ resp.\ $(\calV,\dim)$.

At first we consider the case $\overline{G}=\Sgrp_\calU$.
Indeed, assume for fixed $\sigma=\overline{(\sigma_i)}_{i\in I}\in\Sgrp_\calU$ that $S.\sigma=S$ for all $S\in\calS$. Then take $S=\overline{(S_i)}_{i\in I}$, where $S_i\subseteq\underline{n_i}$ is taken in the following way. For each $k$-cycle $c\subseteq\underline{n_i}$ ($k>1$; here seen as a set) we pick $s_c\in c$ and define $S_i$ by $S_i\cap c=\set{s_c,s_c.\sigma_i^2,\ldots,s_c.\sigma_i^{2(\floor{k/2}-1)}}$ and $S_i\cap\Omega_1(\sigma_i)=\emptyset$. Then $S_i\triangle S_i.\sigma_i=\emptyset$ and $\mu_i(S_i)\geq 1/3\card{\supp(\sigma_i)}$. This means that $S$ is fixed by $\sigma$ if and only if $\supp(\sigma)\coloneqq\overline{(\supp(\sigma_i))}_{i\in I}=\overline{(\emptyset)}$ has measure zero. But this means $\sigma=\id$ in the metric ultraproduct $\Sgrp_\calU$.

Now consider the case $\overline{G}=\PGL_\calU(q)$.
Here, similarly, assume for fixed $g=\overline{(g_i)}_{i\in I}\in\GL_\calU(q)$ that $V.g=V$ for all $V\in\calV$. Then take $V=\overline{(V_i)}_{i\in I}$ in the following way: The linear transformation $g_i$ is a direct sum of Frobenius blocks $F(\chi)$, where $\chi\in k[X]$ runs through all (monic) primary polynomials. For each such block $b\leq k_i^{n_i}$ of dimension $k_b>1$ (here seen as a subspace) of $g_i$ select a cyclic vector $v_b$. Then define $V_i$ by $V_i=\bigoplus_{b,k_b>1}{\gensubsp{v_b,v_b.g_i^2,\ldots,v_b.g_i^{2(\floor{k_b/2}-1)}}}$. Then one observes that $V_i\cap V_i.g_i=0$, so that $\dim_i(V_i+V_i.g_i)-\dim_i(V_i\cap V_i.g_i)=2\dim_i(V_i)$.
This shows that $q_\chi(g)=0$ for all $\chi\in k[X]$ primary of degree $k_\chi>1$. But one observes that, if $q_{(X-\lambda)}(g),q_{(X-\mu)}(g)\geq\varepsilon>0$ for $\lambda\neq\mu$ elements of $k$, we can use the following construction:
Let $e_{i1},\ldots,e_{ik_i}\in V_{X-\lambda}(g_i)$ and $f_{i1},\ldots,f_{ik_i}\in V_{X-\mu}(g_i)$ such that $\lim_\calU{k_i/n_i}=\varepsilon$. Define $V_i\coloneqq\gensubsp{e_{ij}+f_{ij}}[j=1,\ldots,k_i]$ ($i\in I$).
Assume $v\in V_i\cap V_i.g_i$. Then there exist numbers $\alpha_1,\ldots,\alpha_{k_i},\beta_1,\ldots,\beta_{k_i}\in k$ such that 
$$
v=\sum_{j=1}^{k_i}{\alpha_j(e_{ij}+f_{ij})}=\sum_{j=1}^{k_i}{\beta_j(\lambda e_{ij}+\mu f_{ij})}.
$$
This gives that 
$$
\sum_{j=1}^{k_i}{(\alpha_j-\beta_j\lambda)e_{ij}}=\sum_{j=1}^{k_i}{(\beta_j\mu-\alpha_j)f_{ij}},
$$
so that by disjointness of $V_{X-\lambda}(g_i)$ and $V_{X-\mu}(g_i)$ both sides are zero and so, since the $e_{ij},f_{ij}$ ($j=1,\ldots,k_i$) are linearly independent, we get that $\alpha_j-\beta_j\lambda=\beta_j\mu-\alpha_j=0$, so that, since $\lambda\neq\mu$, we obtain $\alpha_j=\beta_j=0$. Hence $v=0$ and $V_i\cap V_i.g_i=0$. But $\lim_\calU{\dim(V_i)/n_i}\geq\varepsilon$, yielding the same contradiction as above.
Therefore we must have $g=\lambda\id$ (for $\lambda\in\finfield_q$; as $k=\finfield_q$ is finite) in the metric ultraproduct $\GL_\calU(q)$, i.e., $\PGL_\calU(q)$ is faithfully represented.

\begin{remark}
	The above statement about $\PGL_\calU(q)$ holds for any such metric ultraproduct of groups $\PGL_{n_i}(k_i)$ where the fields $k_i$ are not restricted with the same proof. Here the kernel of the action $\GL_\calU\to\Aut(\calV,\dim)$ is given by $\prod_\calU{k_i^\times}$ (the algebraic ultraproduct of these groups).
\end{remark}

\begin{remark}\label{rmk:rest_subsp_met_up}
	Hence, if the sequence of subsets $(S_i)_{i\in I}$ resp.\ subspaces $(V_i)_{i\in I}$ is almost stabilized by each element of a subgroup $H$ of $\overline{G}=\Sgrp_\calU$ resp.\ $\overline{G}=\PGL_\calU(q)$ (or of $G=\GL_\calU(q)$), we can restrict $H$ to $S\coloneqq\overline{(S_i)}_{i\in I}$ resp.\ $V\coloneqq\overline{(V_i)}_{i\in I}$.
\end{remark}

\begin{remark}\label{rmk:def_subsp_up}
	For an element $\sigma=\overline{(\sigma_i)}_{i\in I}$ set $\Omega_k(\sigma)\coloneqq\overline{(\Omega_k(\sigma_i))}_{i\in I}\in\calS$ for $k\in\ints_+$. Similarly, for a \emph{semisimple} element $g=\overline{(g_i)}_{i\in I}\in\GL_\calU(q)$ (see the next section for the definition of semisimple elements) set $V_\chi(g)\coloneqq\overline{(V_\chi(g_i))}_{i\in I}\in\calV$ for $\chi\in k[X]$ irreducible. Note that these definitions are independent of the chosen representatives (for the uniqueness of $V_\chi(g)$ we need that $g$ is semisimple, since then $V_\chi(g_i)=\ker(\chi(g_i))$ for a suitable representative $(g_i)_{i\in I}$ of $g$ and $\chi\in k[X]$ irreducible).
\end{remark}

\begin{remark}\label{rmk:bases}
	Call a sequence of subsets $(B_i)_{i\in I}\subseteq k^{n_i}$ a basis of $V\in\calV$ if there is a representative $(V_i)_{i\in I}$ of $V$ such that $B_i$ is a basis of $V_i$ ($i\in I$).
\end{remark}

\begin{remark}\label{rmk:tot_sing}
	Call $V\in\calV$ totally singular if it has a representative $(V_i)_{i\in I}$ such that each $V_i$ is totally singular ($i\in I$).
\end{remark}

\section{Centralizers in $\Sgrp_\calU$, $\GL_\calU(q)$, $\Sp_\calU(q)$, $\GO_\calU(q)$, and $\GU_\calU(q)$}

In this section, we provide tools (Lemmas~\ref{lem:cond_cent_up_cents1} and~\ref{lem:cond_cent_up_cents2}) to compute centralizers of certain elements from the metric ultraproducts $\Sgrp_\calU$ and $\GL_\calU(q)$. We will use this in Section~\ref{sec:cent_in_PGL_U} to compute centralizers of elements in $\PGL_\calU(q)$. We write $\C(g)$ for the \emph{centralizer} of the group element $g$.

\vspace{3mm}

\emph{Centralizers of elements in $G=\Sgrp_\calU$, $\GL_\calU(q)$.} Note that for $\sigma=\overline{(\sigma_i)}_{i\in I}\in\Sgrp_\calU$ resp.\ $g=\overline{(g_i)}_{i\in I}\in\GL_\calU(q)$ we have $\prod_\calU{\C(\sigma_i)}\leq \C(\sigma)$ resp.\ $\prod_\calU{\C(g_i)}\leq \C(g)$ (subsequently, by this notation we mean the metric ultraproduct of subgroups of the $H_i$ ($i\in I$)). In the following lemma, we characterize when the above inclusion is actually an equality in the case of $\Sgrp_\calU$.

\begin{lemma}\label{lem:cond_cent_up_cents1}
	An element $\sigma\in\Sgrp_\calU$ satisfies $\sum_{k\in\ints_+}{q_k(\sigma)}=1$ if and only if for each choice of a representative $(\sigma_i)_{i\in I}$ of $\sigma$, the centralizer $\C(\sigma)$ equals $\prod_\calU{\C(\sigma_i)}$.
\end{lemma}

Before proving Lemma~\ref{lem:cond_cent_up_cents1}, we turn to $\GL_\calU(q)$. An element $g\in\GL_\calU(q)$ is called \emph{semisimple} if it has a representative $(g_i)_{i\in I}$ such that each $g_i\in\GL_{n_i}(q)$ is semisimple, i.e., of order prime to $q$.

\begin{lemma}\label{lem:cond_cent_up_cents2}
	A semisimple element $g\in\GL_\calU(q)$ satisfies $\sum_{\chi\text{ primary}}{q_\chi(g)}=1$ if and only if for each choice of a representative $(g_i)_{i\in I}$ of $g$ where each $g_i$ is semisimple ($i\in I$), the centralizer $\C(g)$ equals $\prod_\calU{\C(g_i)}$.
\end{lemma}

To prove Lemmas~\ref{lem:cond_cent_up_cents1} and~\ref{lem:cond_cent_up_cents2}, we need the following auxiliary result.

\begin{lemma}\label{lem:aux_lrg_inv_subsp}
	The following are true:
	\begin{enumerate}[(i)]
		\item Assume $\sigma\in\Sym(\underline{n})$ is of order $k$ and $S\subseteq\underline{n}$ has normalized counting measure $\mu(S)$. Then $S$ contains a $\sigma$-invariant subset $T$ of measure $\mu(T)\geq 1-k(1-\mu(S))$.
		\item Assume $g\in\GL(V)$ for a $k$-vector space $V$ and that the minimal polynomial of $g$ over $k$ has degree $k$. Assume $U\leq V$. Then there exists a $g$-invariant subspace of $U$ of codimension at most $k\codim(U)$.
	\end{enumerate}
\end{lemma}

\begin{proof}
	(i): Observe that the biggest $\sigma$-invariant subset of $S$ is equal to $T=\bigcap_{i\in\ints}{S.\sigma^i}$. But since $\sigma^k=\id$ by assumption, we see that actually $T=\bigcap_{i=0}^{k-1}{S.\sigma^i}$. Hence, since $\mu(S.\sigma^i)=\mu(S)$ for all $i\in\ints$, we have that $\mu(T)\geq 1-k(1-\mu(S))$.
	
	(ii): Similarly to the above, the biggest $g$-invariant subspace contained in $U$ is $W=\bigcap_{i\in\ints}{U.g^i}$. Now $v\in\bigcap_{i=0}^{k-1}{U.g^i}$ means that $v,\ldots,v.g^{-(k-1)}\in U$. But then $v.g^{-k}=-\frac{1}{a_0}(a_1v.g^{-(k-1)}+\cdots+a_{k-1}v.g^{-1}+v)\in U$, where $\chi=a_0+a_1X+\cdots+a_{k-1}X^{k-1}+X^k$ is the minimal polynomial of $g$. Note that $a_0=(-1)^k\det(g)\neq 0$. This shows that actually $W=\bigcap_{i=0}^{k-1}{U.g^i}$, so that $\codim(W)\leq k\codim(U)$.
\end{proof}

\begin{remark}
	The bounds in Lemma~\ref{lem:aux_lrg_inv_subsp} are sharp. E.g., take $\sigma$ of type $(k^{c_k})$ and set $n=c_k k$. Take $S$ of size $n-s$ such that for precisely $s\leq c_k$ $k$-cycles of $\sigma$, $S$ contains $k-1$ elements of each of them and all elements of the remaining $k$-cycles. Then the set $T$ constructed in Lemma~\ref{lem:aux_lrg_inv_subsp} has size $n-ks$. In (ii) we can use a similar construction. 
\end{remark}

Now we are able to prove the Lemmas~\ref{lem:cond_cent_up_cents1} and~\ref{lem:cond_cent_up_cents2}.

\begin{proof}[Proof of Lemmas~\ref{lem:cond_cent_up_cents1} and~\ref{lem:cond_cent_up_cents2}]
	At first we prove Lemma~\ref{lem:cond_cent_up_cents1}. Assume that $\sigma=\overline{(\sigma_i)}_{i\in I}$, $\tau=\overline{(\tau_i)}_{i\in I}\in\Sgrp_\calU$ commute and assume that $\sum_{k=1}^\infty{q_k(\sigma)}=1$.
	Find a sequence $(N_i)_{i\in I}$ tending to infinity along $\calU$ such that 
	$$
	\lim_\calU\sum_{k=1}^{N_i}{q_k(\sigma_i)}=1\text{ and }\binom{N_i+1}{2}\ell_{\rm H}([\sigma_i,\tau_i])\to_\calU 0.
	$$
	Recall that $C_k(\sigma_i)$ denotes the set of $k$-cycles of the permutation $\sigma_i$.
	Call a $k$-cycle of $\sigma_i$ \emph{bad} if it is not mapped $\sigma_i$-equivariantly to another $k$-cycle of $\sigma_i$ by $\tau_i$. Collect the set of bad $k$-cycles of $\sigma_i$ in $C_k'(\sigma_i)$.
	For each bad $k$-cycle of $\sigma_i$ we get at least one non-fixed point of $[\sigma_i,\tau_i]$, so that $\card{C_k'(\sigma_i)}/n_i\leq\ell_{\rm H}([\sigma_i,\tau_i])$ for all $k\in\ints_+$. Hence, if we change $\tau_i$ such that all bad $k$-cycles of $\sigma_i$ are mapped accurately for $k\leq N_i$ and all $k$-cycles for $k>N_i$ are mapped identically, we get a permutation $\tau_i'$ such that 
	\begin{align*}
	d_{\rm H}(\tau_i,\tau_i') &\leq\frac{1}{n_i}\sum_{k=1}^{N_i}{k\card{C_k'(\sigma_i)}}+\sum_{k=N_i+1}^\infty{q_k(\sigma_i)}\\
	&\leq\left(\sum_{k=1}^{N_i}{k}\right)\ell_{\rm H}([\sigma_i,\tau_i])+\sum_{k=N_i+1}^\infty{q_k(\sigma_i)}\\
	&=\binom{N_i+1}{2}\ell_{\rm H}([\sigma_i,\tau_i])+\sum_{k=N_i+1}^\infty{q_k(\sigma_i)}.
	\end{align*}
	By the assumption $\sum_{k=1}^\infty{q_k(\sigma)}=1$, the last term in the above estimate tends to zero along $\calU$. Hence $\tau=\overline{(\tau_i)}_{i\in I}=\overline{(\tau_i')}_{i\in I}$ and $[\sigma_i,\tau_i']=1$.
	
	Conversely, assume that $\sum_{k=1}^\infty{q_k(\sigma)}<1$. Choose the sequence $(N_i)_{i\in I}$ such that $\lim_\calU\sum_{k=1}^{N_i}{q_k(\sigma_i)}=\sum_{k=1}^\infty{q_k(\sigma)}$ and $\lim_\calU{N_i/n_i}=0$.
	
	For each $i\in I$ change $\sigma_i$ to $\sigma_i'$ such that the $k$-cycles of $\sigma_i'$ are the same as in $\sigma_i$ for $1\leq k\leq N_i$ and the other $k$-cycles of $\sigma_i'$ ($k>N_i$; if they exist) are grouped into one big $K_i$-cycle so that $d_{\rm H}(\sigma_i,\sigma_i')$ is minimal possible. It is easy to see that then still $d_{\rm H}(\sigma_i,\sigma_i')\to_\calU 0$ as $N_i\to_\calU\infty$. Now $\sigma_i'$ eventually has precisely one $K_i$-cycle for $K_i>N_i$. Obtain $\sigma_i''$ by dividing this $K_i$-cycle (if it exists) into two $\floor{K_i/2}$-cycles and at most one fixed point so that $d_{\rm H}(\sigma_i',\sigma_i'')\leq 3/n_i$ is minimal. Note that $K_i/n_i=1-\sum_{k=1}^{N_i}{q_k(\sigma_i)}\to_\calU\varepsilon>0$, so that $\floor{K_i/2}>N_i$ along $\calU$, as $\lim_\calU{N_i/n_i}\to_\calU 0$ by assumption.
	
	Now consider the restriction of the centralizers $\C(\sigma_i')$ and $\C(\sigma_i'')$ to the support of the unique $K_i$-cycle of $\sigma_i'$ (which certainly both fix setwise by the previous inequality). The first group is isomorphic to $\Cycgrp_{K_i}$, whereas the second is isomorphic to $\Cycgrp_{\floor{K_i/2}}\wr \Cycgrp_2$. Taking the metric ultraproducts of these groups restricted to this support (in the sense of Remark~\ref{rmk:rest_subsp_met_up}), we get an abelian group in the first case, and a non-abelian group in the second case. Hence, in at least one case, $\prod_\calU{\C(\sigma_i')}\neq\C(\sigma)$ or $\prod_\calU{\C(\sigma_i'')}\neq\C(\sigma)$.
	
	Now we prove Lemma~\ref{lem:cond_cent_up_cents2}. Assume that $g=\overline{(g_i)}_{i\in I},h=\overline{(h_i)}_{i\in I}\in\GL_\calU(q)$ commute, i.e., $[g,h]=\id$, that $g$ and each $g_i$ ($i\in I$) is semisimple, and assume that 
	$$
	\sum_{\chi\text{ irreducible}}^\infty{q_\chi(g)}=1.
	$$
	Note that semisimplicity implies that for each Frobenius block $F(\chi)$ in the generalized Jordan normal form of $g_i$, $\chi=i^1$ is irreducible. Choose the sequence $(N_i)_{i\in I}$ such that
	$$
	\lim_\calU\sum_{\substack{\chi\text{ irreducible}\\ \deg(\chi)\leq N_i}}{q_\chi(g_i)}=1\text{ and }\left(\sum_{\substack{\chi\text{ irreducible}\\ \deg(\chi)\leq N_i}}{\deg(\chi)}\right)\ell_{\rm rk}([g_i,h_i])\to_\calU 0.
	$$
	
	Define $U_i\coloneqq\ker([g_i,h_i]-\id)$. Fix an irreducible polynomial $\chi\in k[X]$ and apply Lemma~\ref{lem:aux_lrg_inv_subsp}(ii) inside $V\coloneqq V_\chi(g_i)$	to the subspace $U\coloneqq U_i\cap V_\chi(g_i)$ to get a $g_i$-invariant subspace $W=W_{i\chi}\leq U$ such that $\codim_V(W_{i\chi})\leq k_\chi\codim(U_i)$, where $k_\chi=\deg(\chi)$. Note here that $V_\chi(g_i)=\ker(\chi(g_i))$ is unique, since $g_i$ is semisimple. This large-dimensional subspace $W_{i\chi}$ is mapped accurately by $h_i$, as $g_i$ commutes with $h_i$ on it. Define $h_i'$ to be equal to $h_i$ on each $W_{i\chi}$ and complete it on each $V_{i\chi}$ to a map commuting with $g_i$ for $\deg(\chi)\leq N_i$ (here we use semisimplicity of $g_i$). On $V_\chi(g_i)$ with $\deg(\chi)>N_i$ set $h_i'$ to be the identity. As in the proof for $\Sgrp_\calU$ above, it follows that $d_{\rm rk}(h_i,h_i')\to_\calU0$ and $[g_i,h_i']=1$.
	
	Conversely, assume that $\sum_{\chi\text{ irreducible}}^\infty{q_\chi(g)}<1$. Choose the sequence $(N_i)_{i\in I}$ such that 
	$$
	\lim_\calU\sum_{\substack{\chi\text{ irreducible}\\ \deg(\chi)\leq N_i}}^{N_i}{q_\chi(g_i)}=\sum_{\chi\text{ irreducible}}^\infty{q_\chi(g)} \text{ and }\lim_\calU{N_i/n_i}=0.
	$$ 
	For each $i\in I$ change $g_i$ into $g_i'$ such that all Frobenius blocks $F(\chi)$ for $\chi$ irreducible of degree at most $N_i$ are left unchanged and all bigger Frobenius blocks (if there is any such block) are grouped into one big Frobenius block $F(\varphi)$ of size $K_i$ (for $\varphi$ irreducible).
	Define $g_i''$ in the same way, but split the Frobenius block $F(\varphi)$ (if it exists) into two or three blocks, two of which are $F(\phi)$ for $\phi$ irreducible of degree $\floor{K_i/2}$ and, if $K_i$ is odd, one block of size one, which is the identity. Then, as above, the centralizer of $g_i'$ restricted to the large Frobenius block $F(\varphi)$ of it, equals $\C(g_i')\cong (k[X]/(\varphi))^\times$, whereas the centralizer $\C(g_i'')$ restricted to the same subspace is non-abelian (again in the sense of Remark~\ref{rmk:rest_subsp_met_up}). Also one sees that their metric ultraproducts are non-isomorphic, similarly to the case of permutations. The proof is complete.
\end{proof}

\begin{remark}\label{rmk:ctrlzer_up_grp_lt}
	If $G$ is one of $\Sp_\calU(q)$, $\GO_\calU(q)$, or $\GU_\calU(q)$ and a semisimple $g\in G$ is represented by $(g_i)_{i\in I}$ and $\sum_{\chi\text{ irreducible}}{q_\chi(g)}=1$, one can adapt the above argument for $\GL_\calU(q)$ to see that still $\C(g)=\prod_\calU{\C(g_i)}$ when all $g_i$ are semisimple.
	
	Indeed, from Subsection~3.4.2~\S1 of \cite{schneider2019phd} it follows that in the space $W_{i\chi}+W_{i\chi^\ast}$ (where $W_{i\chi},W_{i\chi^\ast}$ are constructed as above) we can still find a big, i.e., almost fulldimensional, $g_i$-invariant non-singular subspace $W_{i\chi,\chi^\ast}'$. Then the form $f_i$ (and $Q_i$ in the orthogonal case in characteristic two) on $W_{i\chi,\chi^\ast}'^\perp\cap(V_\chi(g_i)+V_{\chi^\ast}(g_i))$ and $(W_{i\chi,\chi^\ast}'.h_i)^\perp\cap(V_\chi(g_i)+V_{\chi^\ast}(g_i))$ are isomorphic (which again follows from Subsection~3.4.2~\S1 of \cite{schneider2019phd}), so that we can still complete our partial maps to $h_i'$ ($i\in I$). 
\end{remark}

As a consequence of Lemma~\ref{lem:char_tor_el} together with Lemmas~\ref{lem:cond_cent_up_cents1} and~\ref{lem:cond_cent_up_cents2}, and Remark~\ref{rmk:ctrlzer_up_grp_lt}, we get the following corollary.

\begin{corollary}\label{cor:tor_el_cent_up_cents}
	If $\sigma\in\Sgrp_\calU$ resp.\ a semisimple element $g\in\GL_\calU(q)$, $\Sp_\calU(q)$, $\GO_\calU(q)$, or $\GU_\calU(q)$ is torsion, then $\C(\sigma)$ resp.\ $\C(g)$ is equal to $\prod_\calU{\C(\sigma_i)}$ resp.\ $\prod_\calU{\C(g_i)}$ for each representative $(\sigma_i)_{i\in I}$ resp.\ $(g_i)_{i\in I}$ of $\sigma$ resp.\ $g$, where we require all $g_i$ ($i\in I$) to be semisimple.
\end{corollary}

\section{Centralizers in $\PGL_\calU(q)$, $\PSp_\calU(q)$, $\PGO_\calU(q)$, and $\PGU_\calU(q)$}\label{sec:cent_in_PGL_U}

Now we can deduce the structure of centralizers of semisimple elements from $\PGL_\calU(q)$, i.e., elements that lift to semisimple elements in $\GL_\calU(q)$.
Let $g=\overline{(g_i)}_{i\in I}\in\GL_\calU(q)$ be a semisimple element which maps to $\overline{g}\in\PGL_\calU(q)=\GL_\calU(q)/k^\times$. Here $g_i$ is also assumed to be semisimple ($i\in I$).

Assume that $h=\overline{(h_i)}_{i\in I}\in\GL_\calU(q)$ is such that $[g,h]=\mu\id$ for $\mu\in k^\times$, then
$g^h=\mu g$, so that $q(g)=q(g^h)=q(\mu g)=q(g).\mu$, i.e., $\mu\in\stab_{k^\times}(q(g))$.
Now let $\nu\in\stab_{k^\times}(q(g))\leq k^\times$ be a generator of this cyclic group.

It is now easy to see that the \emph{conformal centralizer} $\C_{\rm conf}(g)\coloneqq\set{h\in\GL_\calU(q)}[\text{there is }\mu\in k^\times\text{ such that }g^h=\mu g]$ is an extension 
$$
\C(g).\stab_{k^\times}(q(g))=\C(g).\gensubgrp{\nu}
$$ 
of $\C(g)$ by $\stab_{k^\times}(q(g))$. Hence $\C(\overline{g})=(\C(g).\gensubgrp{\nu})/k^\times$.

\begin{remark}
	The analog statement of Lemma~\ref{lem:cond_cent_up_cents2} is false in $\PGL_\calU(q)$.
	Indeed, take a semisimple element $\overline{g}\in\PGL_\calU(q)$ such that for a lift $g\in\GL_\calU(q)$ the group $\stab_{k^\times}(q(g))$ is non-trivial. Choose a representative $(g_i)_{i\in I}$ of $g\in\GL_\calU(q)$ such that $q_\chi(g_i)\neq q_\xi(g_i)$ for all $\chi,\xi\in k[X]$ distinct irreducible and $g_i$ is semisimple ($i\in I$). Then $\C(g_i)$ stabilizes each subspace $V_\chi(g_i)=\ker(\chi(g_i))\leq k^{n_i}$. But this means that, if $h\in C\coloneqq\prod_\calU{\C_{\rm conf}(g_i)}$, we have that $g^h=g$, so that $C/k^\times$ is properly contained in $\C(\overline{g})$ (namely, $\C(\overline{g})/(C/k^\times)\cong\stab_{k^\times}(q(g))$, which is non-trivial). 
\end{remark}

\begin{remark}
	For the groups $\PSp_\calU(q)$, $\PGO_\calU(q)$, and $\PGU_\calU(q)$ the same structure for $\C(\overline{g})$ holds, where $\Sp_\calU(q)$, $\GO_\calU(q)$ resp.\ $\GU_\calU(q)$ play the role of $\GL_\calU(q)$. The possible scalars $\mu\in k^\times$ (from $Z$) are restricted to $\mu\in\set{\pm1}$ in the symplectic or orthogonal case, and to $\mu^{q+1}=1$ in the unitary case.
\end{remark}

\section{Double centralizers of torsion elements}\label{sec:dbl_ctrlzr_tor_el_met_ups}

In this section, we compute the double centralizers of (semisimple) torsion elements of the groups $\overline{G}$ of type $\Sgrp_\calU$, $\PGL_\calU(q)$, $\PSp_\calU(q)$, $\PGO_\calU(q)$, and $\PGU_\calU(q)$. Note that for $g\in G$ a group element $\C(\C(g))=\Z(\C(g))$, since $g\in \C(g)$, so that $\C(\C(g))\leq \C(g)$. Set $\C^2(g)\coloneqq\C(\C(g))$ and $\C_{\rm conf}^2(g)\coloneqq\C_{\rm conf}(\C_{\rm conf}(g))$ to be the \emph{double centralizer} resp.\ \emph{double conformal centralizer} of $g$. Here $\C_{\rm conf}(g)\coloneqq\{h\in G\,|\,[g,h]\in\Z(G)\}$.

\subsection{The case $\Sgrp_\calU$}

Let $\sigma=\overline{(\sigma_i)}_{i\in I}\in\Sgrp_\calU=\overline{G}$ be torsion of order $o$. Then $\sum_{k\mid o}{q_k(\sigma)}=1$ by Lemma~\ref{lem:char_tor_el}. By Corollary~\ref{cor:tor_el_cent_up_cents} we have that $\C(\sigma)=\prod_\calU{\C(\sigma_i)}$.
But $\C(\sigma_i)$ has a subgroup 
$$
\prod_{k\mid o}{\Cycgrp_k\wr\Sym(c_k(\sigma_i))}
$$ 
which is dense in it along $\calU$, so that $C\coloneqq \C(\sigma)=\prod_\calU\prod_{k\mid o}{\Cycgrp_k\wr\Sym(c_k(\sigma_i))}$.

At first, for simplicity, assume that $\sigma_i$ is isotypic of type $(k^{c_{ik}})$ (so that $n_i=c_{ik}k$). Assume that $\tau=\overline{(\tau_i)}_{i\in I}\in \Z(C)$ and $\tau_i=(a_{ij}).\varphi_i\in \Cycgrp_k\wr\Sym(c_{ik})$. Assume that $\lim_\calU\card{\supp(\varphi_i)}/c_{ik}=\varepsilon>0$. Then we can conjugate $\varphi_i$ by $\phi_i\in\Sym(c_{ik})\leq \Cycgrp_k\wr\Sym(c_{ik})=\C(\sigma_i)$ such that 
$$
\lim_\calU{d_{\rm H}(\varphi_i\phi_i,\phi_i\varphi_i)}\geq\varepsilon>0.
$$ 
But this leads to the contradiction 
$$
\lim_\calU{d_{\rm H}}(\tau_i\phi_i,\phi_i\tau_i)\geq\varepsilon>0.
$$ 
Hence we may assume that $\varphi_i=\id$, applying a small change to $\tau_i$ along $\calU$ if necessary ($i\in I$).
Now assume that $\lim_\calU{\card{\set{j}[a_{ij}=c]}/c_{ik}}=\varepsilon\in(0,1)$. Then we find permutations $\phi_i\in\Sym(c_{ik})\leq \Cycgrp_k\wr\Sym(c_{ik})=\C(\sigma_i)$ such that $d_{\rm H}(\tau_i,\tau_i^{\phi_i})=\card{\set{j}[a_{ij}\neq a_{ij.\phi_i}]}/c_{ik}\geq\min\set{\varepsilon,1-\varepsilon}>0$. Hence we can assume that all $a_{ij}$ are equal.
This shows that, in this case, $\Z(\C(\sigma))$ is the metric ultraproduct $\prod_\calU{\Cycgrp_k}\cong \Cycgrp_k$ where $\Cycgrp_k$ in the $i$th component is generated by the element $\sigma_i$ itself ($i\in I$).

In the general case, we obtain that 
$$
\Z(\C(\sigma))=\prod_{k\mid o,q_k(\sigma)>0}\prod_\calU{\Cycgrp_k}\cong\prod_{k\mid o,q_k(\sigma)>0}{\Cycgrp_k}.
$$
This holds, because $\sigma\in \C(\sigma)$, so that, when $\tau\in \Z(\C(\sigma))$, it must commute with $\sigma$. But this implies that $\lim_\calU{\card{\Omega_k(\sigma_i)\triangle\Omega_k(\sigma_i).\tau_i}}=0$, so that $\tau$ must stabilize the isotypic components of $\sigma$ (in the sense of Remark~\ref{rmk:rest_subsp_met_up}), and we can apply the above argument.

\subsection{The case $\PGL_\calU(q)$, $\PSp_\calU(q)$, $\PGO_\calU(q)$, and $\PGU_\calU(q)$}\label{subsec:dbl_cent_cl_grps}

Recall that $k=\finfield_q$ when $G$ is $\GL_\calU(q)$, $\Sp_\calU(q)$, or $\GO_\calU(q)$, and $k=\finfield_{q^2}$ when $G=\GU_\calU(q)$. Set $d=1$ in the first three cases and $d=2$ when $G$ is unitary over $\finfield_{q^2}$.

Recall that $Z=k^\times$ when $G=\GL_\calU(q)$, $Z=\set{\pm1}\subseteq k^\times$ when $G=\Sp_\calU(q)$ or $G=\GO_\calU(q)$, and $Z=\set{z\in k^\times}[z^{q+1}=1]\subseteq k^\times=\finfield_{q^2}^\times$ when $G=\GU_\calU(q)$. Also, recall that, if $G$ is not of shape $\GL_\calU(q)$, we have $z^\sigma=z^{-1}$ for $z\in Z$, where $\sigma\colon k\to k$ is the identity in the symplectic and orthogonal case, and the $q$-Frobenius endomorphism $x\mapsto x^q$ when $G=\GU_\calU(q)$. Let $g=\overline{(g_i)}_{i\in I}\in G\leq\GL_\calU(k)$ be semisimple, with $g_i$ ($i\in I$) semisimple such that $\overline{g}\in\overline{G}\leq\PGL_\calU(k)=\GL_\calU(k)/k^\times$ is torsion of order dividing $o$, i.e., there is $\mu\in k^\times$ such that $g^o=\mu\id$. This implies $\mu\in Z$. 
Then 
$$
\sum_{\substack{\chi\text{ irreducible}\\ X^o\equiv\mu\, (\chi)}}{q_\chi(g)}=1
$$
by Lemma~\ref{lem:char_tor_el}. Set $P\coloneqq\set{\chi\in k[X]}[\chi\text{ (monic) irreducible},\, \chi\mid X^o-\mu]$, $T\coloneqq\stab_Z(q(g))$, $K_\chi\coloneqq k[X]/(\chi)$ for $\chi\in k[X]$ irreducible (as in Subsection~3.4.2~\S1 of \cite{schneider2019phd}), and $c_{i\chi}\coloneqq c_\chi(g_i)$ ($i\in I$).
Hence, similarly to the above, we have
$$
\C(g)=\prod_\calU\prod_{\substack{\chi\text{ irreducible}\\  X^o\equiv\mu\, (\chi)\\ q_\chi(g)>0}}{\mathbf M_{c_{i\chi}}(K_\chi)},
$$
the centralizer being computed in $\M_\calU(k)$. Now, by Section~\ref{sec:cent_in_PGL_U} we \enquote*{know} the structure of $\C_{\rm conf}(g)\leq G$. For $\chi\in k[X]$ irreducible consider the $g$-invariant subspace $V\coloneqq V_{\overline{\chi}}(g)\coloneqq\bigoplus_{\xi\in\overline{\chi}}{V_\xi(g)}\in\calV$, where $\overline{\chi}\coloneqq\orb_T(\chi)$ is the orbit of $\chi$ under $T$ (see Remark~\ref{rmk:def_subsp_up} for the definition of $V_\xi(g)\in\calV$). Set $l_\chi\coloneqq\card{\overline{\chi}}$ and $m_\chi\coloneqq\card{T}/l_\chi$. Note that $m_\chi=\card{\stab_T(\chi)}$, and so $m_\chi=\max\set{m\mid\card{T}}[\exists \chi':\chi=\chi'(X^m)]$. The restriction of the action of $\C_{\rm conf}(g)/Z$ to $V_{\overline{\chi}}(g)$ is given by 
$$
\left(\left(\prod_\calU\prod_{\xi\in\overline{\chi}}{\C\left(\rest{g}_{V_\xi(g)}\right)}\right)\rtimes T\right)/Z.
$$
We will explain this below.

\vspace{3mm}

\emph{Definition of the action of $T$.} In this situation, $t\in T\leq Z\leq k^\times$ acts as the map $\varphi_t$ which is constructed as follows: Find $K_\xi$-bases $(B_{\xi,i})_{i\in I}$ of each $V_\xi(g)$ ($\xi\in\overline{\chi}$ for all representatives $\chi$ of orbits of the action of $T$ on the irreducible polynomials; see Remark~\ref{rmk:bases}) and compatible bijections $\alpha_{\xi_1,\xi_2,i}\colon B_{\xi_1,i}\to B_{\xi_2,i}$ ($i\in I$; i.e., $\alpha_{\xi_2,\xi_3,i}\circ\alpha_{\xi_1,\xi_2,i}=\alpha_{\xi_1,\xi_3,i}$ for all $\xi_1,\xi_2,\xi_3\in\overline{\chi}$, all $\chi$, and all $i\in I$). If $G$ comes from groups preserving a form, we still find bijections $\bullet^\ast\colon B_{\xi,i}\to B_{\xi^\ast,i}$ such that $b^{\ast\ast}=b$, the pairing $f_i$ restricted to $K_\xi b\times K_{\xi^\ast}b^\ast\to k$ is non-singular, the pairing $f_i$ restricted to $K_\xi b\times K_{\xi'}b'$ is zero for all $b\in B_{\xi,i}$, $b'\in B_{\xi',i}$, $b'\neq b^\ast$, and such that $\bullet^\ast$ commutes with the maps $\alpha_{\xi_1,\xi_2,i}$ ($i\in I$). Such bases exist by the classification in Subsection~3.4.2~\S1 of \cite{schneider2019phd}. The last condition can be fulfilled, since $(\xi.t)^\ast=\xi^\ast.t$ for all $\xi\in k[X]$ and $t\in T$ by Remark~\ref{rmk:comm_*_Z}.
Define $\varphi_{t,i}$ by $\rest{\varphi_{t,i}}_{\gensubsp{B_{\xi,i}}_{K_\xi}}\colon \gensubsp{B_{\xi,i}}_{K_\xi}\to \gensubsp{B_{\xi.t,i}}_{K_{\xi.t}}$ to be the field isomorphism $\varphi_{\xi,t}\colon K_\xi=k[X]/(\xi)\to K_{\xi.t}=k[X]/(\xi.t)$; $\overline{X}\mapsto t\overline{X}$ applied to each $K_\xi$-multiple of a basis vector in $B_{\xi,i}$, i.e.,
$$
\varphi_{t,i}\left(\sum_{b\in B_{\xi,i}}{\lambda_b b}\right)=\sum_{b\in B_{\xi.t,i}}{\varphi_{\xi,t}(\lambda_b)\alpha_{\xi,\xi.t,i}(b)}.
$$
Doing this for all representatives $\chi$ of orbits of the action of $T$ on the irreducible polynomials $\chi\in k[X]$ with $\chi\mid X^o-\mu$, this defines, up to a small error in the rank metric, a map $\varphi_{t,i}\colon k^{n_i}\to k^{n_i}$, so set $\varphi_t$ to be $\overline{(\varphi_{t,i})}_{i\in I}$. 

\vspace{3mm}

\emph{The action of $T$ preserves the forms $f_i$ (and $Q_i$; $i\in I$).} Assume $G$ is not $\GL_\calU(q)$. Then one verifies that $T$ preserves the forms $f_i$ ($i\in I$): According to Subsection~3.4.2~\S1 of \cite{schneider2019phd}, for $b\in B_{\xi,i}$ the form $\rest{f_i}_{U\times U^\ast}\colon U\times U^\ast\coloneqq K_\xi b\times K_{\xi^\ast}b^\ast\to k$ is given as $U\times U^\ast\cong K_\xi\times K_{\xi^\ast}\ni(u,v)\mapsto\beta\tr_{K_\xi/k}(uv^\alpha)$ (where $\alpha\colon K_{\xi^\ast}\to K_\xi$ as remarked in Remark~3.33 of \cite{schneider2019phd}, noting that $K_\xi=R_\xi$ as $e=1$, since $g$ is semisimple, and $\beta$ is either one or a standard non-square in $k^\times$; the latter is only needed in Case~3.1 of Subsection~3.4.2~\S1 of \cite{schneider2019phd} when $G$ is orthogonal and $b=b^\ast$; but we can even neglect this case by Remark~3.38 of \cite{schneider2019phd}; so $\beta=1$). In Case~3.2 of Subsection~3.4.2~\S1 of \cite{schneider2019phd}, i.e., $p=2$, so $f_i$ is alternating and thus $b\neq b^\ast$, we can assume additionally that $Q(\lambda b+\mu b^\ast)=\lambda\mu\in k$, as all but at most one irreducible block have this shape $W(1)$ (cf.\ \cite[page~8 and Theorem~3.1]{gonshawliebeckobrien2017unipotent}).

Hence for $(u,v)\in K_\xi\times K_{\xi^\ast}\cong U\times U^\ast$ we obtain 
\begin{align*}
f_i(u.t,v.t)&=\tr_{K_{\xi.t}/k}(\varphi_t(u)\varphi_t(v)^\alpha)=\tr_{K_{\xi.t}/k}(\varphi_t(u)\varphi_t(v^\alpha))\\
&=\tr_{K_{\xi.t}/k}(\varphi_t(uv^\alpha))=\tr_{K_\xi/k}(uv^\alpha)=f_i(u,v).
\end{align*}
This holds, since the action of $T$ commutes with $\alpha$ and $\varphi_t$ is a field isomorphism. The former is verified as follows: Let $v\in k[X]$. Then $\varphi_t(v(\overline{X}))^\alpha=v(t\overline{X})^\alpha=v^\sigma(t^\sigma\overline{X}^{-1})=v^\sigma(t^{-1}\overline{X}^{-1})=\varphi_t(v^\sigma(\overline{X}^{-1}))=\varphi_t(v^\alpha)$, as desired, since by definition of $Z$ we have $t^\sigma=t^{-1}\in Z$. Here $\overline{X}$ is the image of $X$ in $K_\xi=k[X]/(\xi)$.

\vspace{3mm}

Now let us fix $h\in\C_{\rm conf}^2(g)$. We want to understand the shape of $h$.

\vspace{3mm}

\emph{Step~1: $h$ stabilizes each $V_\chi(g)$ ($\chi\in k[X]$ irreducible).} Assume that $h\in \C^2_{\rm conf}(\rest{g}_V)\leq\C_{\rm conf}(\rest{g}_V)$ does not stabilize each subspace $V_\xi(g)$ of $V$ ($\xi\in\overline{\chi}$). Write $\overline{\chi}=\set{\xi_1,\ldots,\xi_l}$ and assume that $V_{\xi_1}(g).h=V_{\xi_2}(g)$. Take $f=(M_1,M_2,\ast,\ldots,\ast)\in \C(\rest{g}_V)\leq\C_{\rm conf}(\rest{g}_V)$, where the $j$th component of $f$ acts on $V_{\xi_j}(g)$ ($j=1,\ldots,l$), then 
$$
f^h=h^{-1}fh=(\ast,M_1^h,\ast,\ldots,\ast).
$$

Now there are three cases according to the classification in Subsection~3.4.2~\S1 of \cite{schneider2019phd}: If $G=\GL_\calU(q)$, we can take $M_2=1_{V_{\xi_2}(g)}$ and $M_1$ far away from $k^\times\id_{V_{\xi_1}(g)}$. Then $[f,h]=(\ast,M_1^h,\ast,\ldots,\ast)$ is far away from $k^\times\id_V$. If $G$ is one of $\Sp_\calU(q)$, $\GO_\calU(q)$, or $\GU_\calU(q)$, $\xi_1$ is not self-dual and $\xi_2\neq\xi_1^\ast$, we can do the same as before. When $\xi_1^\ast=\xi_2$ in this case, we must have $M_2=(M_1^{-\sigma})^\top$, so that $[f,h]=(\ast,M_1^{\sigma\top} M_1^h,\ast,\ldots,\ast)$. Again we can choose $M_1\in\prod_\calU{\GL_{c_{i\xi_1}}(K_{\xi_1})}$ such that $(M_1^\sigma)^\top M_1^h$ is far away from $Z$. In the last case, $\xi_1=\xi_1^\ast$ is self-dual. Then again $M_1$ and $M_2$ are independent of each other and we can choose $M_2=1_{V_{\xi_2}(g)}$. The only restriction on $M_1$ is that it lies in $\prod_\calU{\GU_{c_{i\xi_1}}(K_{\xi_1})}$ if $\xi_1\neq X\pm1$ or $G$ is $\GU_\calU(q)$ (see Case~2 of Subsection~3.4.2~\S1 of \cite{schneider2019phd}) resp.\ $M_1\in\prod_\calU{X_{c_{i\xi_1}}(k)}$ in the opposite case when $\xi_1=X\pm1$, where $G=X_\calU(q)$ ($X=\Sp$ or $\GO$; see Case~3.1 of Subsection~3.4.2~\S1), so again we can choose $M_1$ such that $[f,h]=(\ast,M_1^h,\ast,\ldots,\ast)$ is far away from $Z$. In all cases, we get a contradiction. This shows that $h\in \C^2_{\rm conf}(g)$ fixes each $V_\chi(g)\in\calV$ ($\chi\in k[X]$ irreducible). 

\vspace{3mm}

Assume now that $\rest{h}_{V_\chi(g)}=M.\alpha$, where $\alpha$ corresponds to an element of $T^{l_\chi}=\set{t^{l_\chi}}[t\in T]$ which induces a non-trivial field automorphism on $K_\chi$. 

\vspace{3mm}

\emph{Step~2: The automorphism $\alpha$ equals the identity $\id_{K_\chi}$.} Then for $\lambda\in K_\chi$ we have $(\lambda\id)^h=\lambda^\alpha\id=(\lambda^\alpha\lambda^{-1})\lambda\id$. This implies that for all $\lambda\in K_\chi^\times$ stabilizing the forms $f_i$ (or $Q_i$; $i\in I$) on $V_\chi(g)$ we have $\lambda^\alpha\lambda^{-1}\in Z\leq k^\times$. When $G=\GL_\calU(q)$ or $\chi$ is not self-dual, there is no restriction on $\lambda$ (of course, if $G$ is one of $\Sp_\calU(q)$, $\GO_\calU(q)$, or $\GU_\calU(q)$, then if $h$ acts as $M$ on $V_\chi(g)$, it must act as $(M^{-\sigma})^\top$ on $V_{\chi^\ast}(g)$). Hence, in this case, for each $\lambda^\times\in K_\chi$ there exists $\kappa_\lambda\in k^\times$ such that $\lambda^\alpha\lambda^{-1}=\kappa_\lambda$. However, then every vector $\lambda\in K_\chi$ is an eigenvector of the $k$-linear map $\alpha$, which forces $\alpha=\id_{K_\chi}$, since $1\in K_\chi$ is fixed, a contradiction. 


In the opposite case, $G$ is one of $\Sp_\calU(q)$, $\GO_\calU(q)$, or $\GU_\calU(q)$ and $\chi$ is self-dual. Then we are in Case~2 and~3 of Subsection~3.4.2~\S1 of \cite{schneider2019phd}. Let $\tau\colon K_\chi\to K_\chi$ be the map defined there, i.e., $\rest{\tau}_k=\sigma$ and $\tau\colon \lambda\mapsto \lambda^{-1}$, where $\lambda\in K_\chi$ is the root of $\chi$. Then $\tau^2=\id_{K_\chi}$ and $\tau=\id_{K_\chi}$ if and only if we are in Case~3 of Subsection~3.4.2~\S1 of \cite{schneider2019phd}. Here, if we are in Case~2, $\rest{\C(g)}_{V_\chi(g)}$ is an ultraproduct of unitary groups over the field $K_\chi$ equipped with the involution $\tau$. In Case~3, $\rest{\C(g)}_{V_\chi(g)}$ is an ultraproduct of symplectic resp.\ orthogonal groups over $K_\chi=k$. We proceed as follows: Find totally singular $K_\chi$-subspaces $U=\overline{(U_i)}_{i\in I},U'=\overline{(U'_i)}_{i\in I},U''=\overline{(U''_i)}_{i\in I}\in\calV$ of $V_\chi(g)$ (in the sense of Remark~\ref{rmk:tot_sing}) such that $U\oplus U'=U\oplus U''=V_\chi(g)$, $U'\cap U''=0$ and $\dim(U)=\dim(U')=\dim(U'')=\dim(V_\chi(g))/2$. W.l.o.g., we may assume that $\dim_{K_\chi}(U_i)=\dim_{K_\chi}(U'_i)=\dim_{K_\chi}(U''_i)$ and that the restrictions $\rest{f_i}_{U_i\times U'_i}$ and $\rest{f_i}_{U_i\times U''_i}$ are non-degenerate ($i\in I$; as we may by modifying $U_i$, $U'_i$, and $U''_i$ a little if necessary). Then define $f'=\overline{(f'_i)}_{i\in I},f''=\overline{(f''_i)}_{i\in I}\in\C(g)\leq G$ such that
$f'_i$ and $f''_i$ act $F(\varphi)$-isotypically on $U_i$ and such that $f'_i$ resp.\ $f''_i$ act $F(\varphi^\ast)$-isotypically on $U'_i$ resp.\ $U''_i$ ($i\in I$) for a fixed irreducible polynomial $\varphi\in K_\chi[X]$ which is not self-dual with respect to $\tau$. Then $\rest{f'^h}_{V_\chi(g)}=z'\rest{f}_{V_\chi(g)}'$ and $\rest{f''^h}_{V_\chi(g)}=z''\rest{f''}_{V_\chi(g)}$ for $z',z''\in Z$. Note that $q_{\varphi.z'^{-1}}(\rest{z'f'}_{V_\chi(g)})=q_\varphi(\rest{f'}_{V_\chi(g)})=1/2$ and $q_{\varphi.z''^{-1}}(\rest{z''f''}_{V_\chi(g)})=q_\varphi(\rest{f''}_{V_\chi(g)})=1/2$, and $\varphi.z'^{-1}$ and $\varphi.z''^{-1}$ are both also not self-dual, since $\varphi\in K_\chi[X]$ is not self-dual and $z'^{-1},z''^{-1}\in Z$, so that $z'^{-\tau}=z'^{-\sigma}=z'$ and $z''^{-\tau}=z''^{-\sigma}=z''$, whence, e.g., $(\varphi.z'^{-1})^\ast=\varphi^\ast.z'^{-1}\neq\varphi.z'^{-1}$. Then $h$ must stabilize the decompositions $V_\chi(g)=U\oplus U'=U\oplus U''$, so it must stabilize $U$. But on the $h$-invariant totally isotropic subspace $U$, we can do the same argument as above for $G=\GL_\calU(q)$ to see that $\alpha=\id_{K_\chi}$.

\vspace{3mm}

Hence we have obtained that $\rest{h}_{V_\chi(g)}=M\in \prod_\calU\M_{c_{i\chi}}(K_\chi)$, so that $h\in\C(g)$.

\vspace{3mm}

\emph{Step~3: We have that $\rest{h}_{V_\chi(g)}=M=\lambda\id$ for $\lambda\in K_\chi$.}
According to Subsection~3.4.2~\S1 of \cite{schneider2019phd}, we can find $V_i$ ($i\in I$) such that $\overline{(V_i)}_{i\in I}=V_\chi(g)$ such that either all $V_i$ are totally singular (Case~1 of Subsection~3.4.2~\S1 of \cite{schneider2019phd}; if $\chi$ is not self-dual; this includes the case $G=\GL_\calU(q)$) or $H_i$ preserves a unitary form (Case~2) or a symplectic or orthogonal form (Case~3) over $K_\chi$ on $V_i$ ($i\in I$).
Note from the classification in Subsection~3.4.2~\S1 of \cite{schneider2019phd} that orthogonally indecomposable blocks involving a Frobenius block of size $\geq 2$ are non-central in the ambient projective linear classical group. This shows that $q_\xi(M)=0$ for all $\xi\in K_\chi[X]$ of degree $\geq 2$. Assume now that there exist distinct $\lambda,\mu\in K_\chi^\times$ such that $q_{X-\lambda}(M),q_{X-\mu}(M)\geq\varepsilon>0$. If $G=\GL_\calU(q)$ or we are in Case~1, $F(X-\lambda)\oplus F(X-\mu)=\diag(\lambda,\mu)\in\GL_2(K_\chi)$ is mapped to a non-central element in $\PGL_2(K_\chi)$, so that by the assumption, since we have \enquote*{many} of these blocks, $\rest{h}_{V_\chi(g)}$ would not commute modulo scalars with all of $\rest{\C(g)}_{V_\chi(g)}\cong\prod_\calU\M_{c_{i\chi}}(K_\chi)$. In Case~2, we use the same argument for a block of shape $\diag(\lambda,\lambda^{-\tau},\mu,\mu^{-\tau})$ acting on a four-dimensional $(K_\chi,\tau)$-unitary space. In Case~3, we use the same argument with a block $\diag(\lambda,\lambda^{-1},\mu,\mu^{-1})$ acting on a four-dimensional symplectic or orthogonal space. In total we get that $M=\lambda\id$ for $\lambda\in K_\chi$. If we are in Case~2 of Subsection~3.4.2~\S1 of \cite{schneider2019phd}, we have the additional assumption that $\N_\tau(\lambda)=1$, where $\N_\tau\colon K_\chi\to K_{\chi,\tau}$ is the norm defined there. In Case~3 of Subsection~3.4.2~\S1 of \cite{schneider2019phd}, we must have $\lambda^2=1$.

\vspace{3mm}

\emph{Step~4: The precise shape of $C\coloneqq\C_{\rm conf}^2(g)$.} We know now that $\rest{h}_{V_\chi(g)}=\lambda_\chi(h)\id$ for each irreducible $\chi\in k[X]$ and so $h$ commutes with all of $\C(g)$. In order that $h\in\C_{\rm conf}^2(g)$, we still need to check that $[h,T]\subseteq Z$. Now choose $t\in T$ to be a generator and $z\in Z$ and assume that $h$ $z$-commutes with $t$, i.e., $[h,\varphi_t]=z\id$. Let $\chi\in k[X]$ run through a system of representatives of the orbits of the action of $T$ and $\bullet^\ast$ on the irreducible polynomials (the action of $\bullet^\ast$ is only used when $G$ is not $\GL_\calU(q)$). This means $zh=h^{\varphi_t}$, so since $\rest{h}_{V_\chi(g)}=\lambda_\chi(h)\id$, we must have $\rest{h}_{V_{\chi.t}(g)}=z^{-1}\varphi_{\chi,t}(\lambda_\chi(h))\id_{V_{\chi.t}}$, so that $h$ is determined on all of $V=V_{\overline{\chi}}(g)$ by $\lambda_\chi(h)$. In this situation, the only condition that needs to be satisfied is that $\rest{h}_{V_\chi(g)}=\rest{h}_{V_{\chi.t^{l_\chi}}(g)}=\lambda_\chi(h)\id=z^{-l_\chi}\varphi_{\chi,t^{l_\chi}}(\lambda_\chi(h))\id$. Note that 
$$
\varphi_{\chi,t^l_\chi}\colon K_\chi\cong\finfield_{q^{dk_\chi}}\to K_\chi\cong\finfield_{q^{dk_\chi}}\quad (d=1,2)
$$ 
is given by $x\mapsto q^{dk_\chi/m_\chi}$, so that the previous condition becomes 
\begin{equation}\label{eq:cond}
z^{l_\chi}=(\lambda_\chi(h))^{q^{dk_\chi/m_\chi}-1}.
\end{equation}
Hence we can write $C$ as follows. When $G=\GL_\calU(q)$, we have

\begin{equation}\label{eq:C^2_lin_cs}
C=\lrset{h=\bigoplus_{\substack{\chi\text{ irreducible}\\ X^o\equiv\mu\, (\chi)\\ q_\chi(g)>0}}\lambda_\chi(h)\id_{V_\chi(g)}}[\exists z\in Z: \lambda_{\chi.t}(h)=z^{-1}\varphi_{\chi,t}(\lambda_\chi(h))\text{ for all }\chi].
\end{equation}

Here the condition from Equation~\eqref{eq:C^2_lin_cs} is equivalent to Equation~\eqref{eq:cond} for $\chi$ running through a system of representatives of the action of $T$ on the set $P\coloneqq\set{\chi\in k[X]\text{ irreducible}}[\chi\text{ divides } X^o-\mu\text{ and }q_\chi(g)>0]$. For $G$ one of $\Sp_\calU(q)$, $\GO_\calU(q)$, or $\GU_\calU(q)$ we have
\begin{equation}\label{eq:C^2_bil_un_cs}
C=\lrset{\bigoplus_{\substack{\chi\text{ irreducible}\\ X^o\equiv\mu\, (\chi)\\ q_\chi(g)>0}}\lambda_\chi(h)\id_{V_\chi(g)}}[R].
\end{equation}
where the condition $R$ is that there exists $z\in Z$ such that $\lambda_{\chi.t}(h)=z^{-1}\varphi_{\chi,t}(\lambda_\chi(h))$ for all $\chi\in P$ (as in the previous case) and $\lambda_\chi(h)\lambda_{\chi^\ast}(h)^\alpha=1$ for all $\chi\in P$, where $\alpha\colon R_{\chi^\ast}=K_{\chi^\ast}\to R_\chi=K_\chi$ is defined as in Remark~3.33 of \cite{schneider2019phd}. If $G$ is one of $\Sp_\calU(q)$ or $\GO_\calU(q)$ and $\chi=\chi^\ast\neq X\pm 1$ is self-dual, this means $k_\chi$ is even and $\lambda_\chi(h)^{q^{k_\chi/2}+1}=1$. Also, in this case, if $\chi=X\pm1$ it means $\lambda_\chi(h)^2=1$. If $G=\GU_\calU(q)$ and $\chi=\chi^\ast$, this means that $k_\chi$ is odd (since $\alpha$ needs to induce $\sigma:x\mapsto x^q$ on $k=\finfield_{q^2}$) and $\lambda_\chi(h)^{q^{k_\chi}+1}=1$.

\section{Distinction of metric ultraproducts}

Now we want to distinguish all (simple) metric ultraproducts $\overline{G}=\overline{X}_\calU(q)$ for distinct pairs $(X,q)$, where $X\in\set{\GL,\Sp,\GO,\GU}$ and $q$ is a prime power (all but $\PSp_{\calU_1}(q)$ and $\PGO_{\calU_2}(q)$ as mentioned in Theorem~\ref{thm:non_iso_main}).
For a group $H$ define the quantity 
$$
e_H(o)\coloneqq\max_{h\in H: h^o=1_H}{\exp(\C^2(h))}.
$$
Clearly, when $H\cong L$, we have $e_H(o)=e_L(o)$ for all values $o\in\ints_+$. Our strategy is to compute $e_H(o)$ for the groups $H=\overline{G}$, where $\overline{G}=\overline{X}_\calU(q)$ as above, for certain values of $o$ to distinguish these groups (with the only exception: $\PSp_{\calU_1}(q)\cong\PGO_{\calU_2}(q)$?).

\subsection{Computation of $e_{\overline{G}}(o)$ when $\gcd\set{o,p}=\gcd\set{o,\card{Z}}=1$}

If $o$ is coprime to $\card{Z}$ (and by semisimplicity of $g\in G$ coprime to $p$), from Subsection~\ref{subsec:dbl_cent_cl_grps} we can compute $e_{\overline{G}}(o)$. 
Note that in this situation, when $g^o=\mu\in Z$, we can replace $g$ by $g'=\lambda g\in G$ such that $g'^o=1$, choosing $\lambda\in Z$ such that $\lambda^o=\mu^{-1}$, since the homomorphism $Z\to Z$; $x\mapsto x^o$ is then bijective. So assume, w.l.o.g., $g^o=1$. Then $P\subseteq Q\coloneqq\set{\chi\in k[X]\text{ irreducible}}[\chi\text{ divides }X^o-1]$. 

\vspace{3mm}

\emph{The case $G=\GL_\calU(q)$.}  From Equation~\eqref{eq:C^2_lin_cs} we see that, the bigger the group $T$ is, for an element $h\in\C^2_{\rm conf}(g)$, the more restrictions are imposed to the scalars $\lambda_\chi(h)\in K_\chi^\times$ ($\chi\in P$). Also, the bigger the set $P$ is, the \enquote*{bigger} is the group $\C^2_{\rm conf}(g)$, i.e., there are more components. Hence, to optimize the exponent of $\C^2(\overline{g})=\C^2_{\rm conf}(g)/Z$, we choose $g$ such that $P=Q$ and $0<q_{X-1}(g)\neq q_\chi(g)>0$ for all $\chi\in P\setminus\set{X-1}$. Namely, then $T=\stab_Z(q(g))$ must fix the polynomial $X-1$, so that we must have $T=\trivgrp$. Set 
$$
f_q(o)\coloneqq\min\set{q^e-1}[o\text{ divides }q^e-1].
$$

Equation~\eqref{eq:C^2_lin_cs} then gives
\begin{equation}\label{eq:frmla_e_o_cprme_lin_cs}
e_{\overline{G}}(o)=\exp(\C^2(\overline{g}))=\exp(\C^2_{\rm conf}(g)/Z)=\begin{cases}
1 & \text{if } o=1\\
f_q(o) & \text{if } o>1
\end{cases}.
\end{equation}
Let us demonstrate Equation~\eqref{eq:frmla_e_o_cprme_lin_cs}. The first equality in it holds by the previous argument. When $o=1$, we have $\overline{g}=1_{\overline{G}}$ and so 
$$
\C^2(\overline{g})=\Z(\overline{G})=\trivgrp, 
$$
so that $e_{\overline{G}}(1)=1$. Now assume $o>1$. For each $\chi\in P$, if $\lambda\in\overline{k}^\times$ is a root of $\chi$, the condition that $\chi\mid X^o-1$ is equivalent to $\lambda^o=1$. Also $K_\chi=k[\lambda]$. Let $\mu\in\overline{k}^\times$ be an element of order $o$ with minimal polynomial $\xi\in k[X]$. Then, if $\lambda$ is a root of $\chi\in P$, we must have $\lambda^o=1$ and thus $\lambda=\mu^f$ for some $f\in\nats$. Hence $K_\chi=k[\lambda]=k[\mu^f]\subseteq k[\mu]=K_\xi$, so that in Equation~\eqref{eq:C^2_lin_cs} we have $\ord(\lambda_\chi(h))\mid\card{K_\chi^\times}\mid\card{K_\xi^\times}=f_q(o)$. This shows that 
\begin{align*}
\exp(\C^2(\overline{g}))\mid\exp(\C^2_{\rm conf}(g))&\mid\lcm\set{\card{K_\chi^\times}}[\chi\text{ divides }X^o-1]\\
&\quad=\card{K_\xi^\times}=\card{k[\mu]^\times}=f_q(o).
\end{align*}
To show the equality $\exp(\C^2(\overline{g}))=f_q(o)$, take $h\in\C^2_{\rm conf}(g)$ such that $\lambda_{X-1}(h)=1$ and $\lambda_\xi(h)$ has order $f_q(o)=\card{K_\xi^\times}$ in $\overline{k}^\times$. Then, when $\overline{h}^l=1_{\overline{G}}$, we must have $h^l\in Z$. But $\lambda_{X-1}(h)^l=1$, so that, since $q_{X-1}(g)>0$, it follows that $h^l=1_G$. Then $\lambda_\xi(h)^l=1$, so that $\exp(\C^2(\overline{g}))\geq l\geq\ord(\lambda_\xi(h))=f_q(o)$. This completes the proof.

\vspace{3mm}

\emph{The case $G=\Sp_\calU(q)$ or $\GO_\calU(q)$.} As in the linear case, Equation~\eqref{eq:C^2_bil_un_cs} shows that the optimal exponent of $\C^2(\overline{g})$ is obtained when $P=Q$ and $0<q_{X-1}(g)\neq q_\chi(g)>0$ for all $\chi\in P\setminus\set{X-1}$, so that $T=\trivgrp$. Set 
\begin{equation}\label{eq:def_f'_q(o)}
f'_q(o)\coloneqq\begin{cases}
q^{e/2}+1 & \text{ if } f_q(o)=q^e-1, e\text{ is even and }o\mid q^{e/2}+1\\
f_q(o) & \text{otherwise}
\end{cases}.
\end{equation}
Equation~\eqref{eq:C^2_bil_un_cs} then gives
\begin{equation}\label{eq:frmla_e_o_cprime_bil_cs}
e_{\overline{G}}(o)=\exp(\C^2(\overline{g}))=\exp(\C^2_{\rm conf}(g)/Z)=\begin{cases}
1 & \text{if } o=1\\
2 & \text{if } o=2\\
f'_q(o) & \text{if } o>2
\end{cases}.
\end{equation}
We demonstrate Equation~\eqref{eq:frmla_e_o_cprime_bil_cs}. If $o=1$, we obtain, as in the linear case, that $\C^2(\overline{g})=\trivgrp$ and so $e_{\overline{G}}(1)=1$. If $o=2$, $g^2=1$ and so $P=\set{X-1,X+1}$. From Equation~\eqref{eq:C^2_bil_un_cs} we see that, if $h\in\C^2_{\rm conf}(g)$, we have $\lambda_{X-1}(h)^2=\lambda_{X+1}(h)^2=1$, so that $h^2=1$. Also, defining $h$ by $\lambda_{X-1}(h)\coloneqq 1$ and $\lambda_{X+1}(h)=-1$, we obtain $h\notin Z$, so $\ord_{\overline{G}}(\overline{h})=e_{\overline{G}}(2)=2$ ($1\neq -1$, since the case $p=2$ does not occur due to the condition $\gcd\set{o,p}=1$). Assume now that $o>2$. As in the linear case, for each $\chi\in P$, if $\lambda\in\overline{k}$ is a root of $\chi$, the condition that $\chi\mid X^o-1$ is equivalent to $\lambda^o-1$. Choose $\mu\in\overline{k}^\times$ of order $o$ and let $\xi\in P$ be its minimal polynomial. Then, as previously, if $\lambda$ is a root of $\chi\in P$, we have $\lambda=\mu^f$ for some $f\in\nats$. There are two cases: 

In the first case, $\xi$ is not self-dual. This means that $\mu$ and $\mu^{-1}$ are not conjugate in $K_\xi/k$. If they were conjugate, say by an automorphism $\alpha$, i.e., $\mu^\alpha=\mu^{-1}$, then $\alpha\in\Gal(K_\xi/k)$ needs to be the unique involution (since $\mu\neq\mu^{-1}$ as $o>2$) given by $x\mapsto x^{q^{k_\xi/2}}$; in particular, $e=k_\xi$ would need to be even. Hence this case is equivalent to either $e=k_\xi$ being odd or $\mu\mu^\alpha=\mu^{q^{k_\xi/2}+1}\neq 1$, i.e., $o\nmid q^{e/2}+1=q^{k_\xi/2}+1$. This is precisely the opposite of the first case in Equation~\eqref{eq:def_f'_q(o)}. Here for an element $h\in\C^2_{\rm conf}(g)$ we can choose $\lambda_\xi(h)\in K_\xi^\times=k[\mu]^\times$ arbitrarily ($\lambda_{\xi^\ast}(h)$ is then determined by $\lambda_\xi(h)$). Arguing as in the linear case, we obtain $\exp(\C^2(\overline{g}))=f_q(o)$. Indeed, for $h\in\C^2_{\rm conf}(g)$, as above, $\ord(\lambda_\chi(h))\mid f_q(o)$ and defining $h$ such that $\lambda_{X-1}(h)=1$ and $\lambda_\xi(h)$ has order $f_q(o)$, we see that $\ord_{\overline{G}}(\overline{h})=e_{\overline{G}}(o)=f_q(o)$. 

In the opposite case, $\xi$ is self-dual and $\xi\neq X\pm1$ as $o>2$. Then $e=k_\xi$ needs to be even and
$$
\mu\mu^\alpha=\mu^{q^{k_\xi/2}+1}=\lambda_\xi(h)\lambda_\xi(h)^\alpha=\lambda_\xi(h)^{q^{k_\xi/2}+1}=1,
$$ 
where $\alpha$ is the involution $x\mapsto x^{q^{k_\xi/2}}$ of $K_\xi\cong\finfield_{q^{k_\xi}}$ from Subsection~3.4.2~\S1 of \cite{schneider2019phd} Case~2. This means that $o\mid q^{e/2}+1$ and we are in the first case of Equation~\eqref{eq:def_f'_q(o)}. Note that for each $\chi\in P$ the map $\alpha$ restricts to an automorphism of each $K_\chi\subseteq K_\xi$ of order dividing two (as all the fields are finite). Then $\rest{\alpha}_{K_\chi}=\id$ if and only if $k_\chi\mid k_\xi/2$, and $\rest{\alpha}_{K_\chi}$ is the unique involution of $K_\chi$ if $k_\xi/k_\chi$ is odd.
Now, if $\lambda\in\overline{k}^\times$ is a root of $\chi$, then in the first case $\lambda^2=1$, and in the second case $\lambda^{q^{k_\chi/2}+1}=1$; so all $\chi\in P$ are self-dual. Hence, if $h\in\C^2_{\rm conf}(g)$, for each $\chi\in P$ one of $\lambda_\chi(h)^2=1$ or $\lambda_\chi(h)^{q^{k_\chi/2}+1}=1$ must hold. But $2\mid q^{k_\xi/2}+1$ if $p>2$, and $q^{k_\chi/2}+1\mid q^{k_\xi/2}+1$ in the second case, since $k_\xi/k_\chi$ is then odd. Hence $\exp(\C^2_{\rm conf}(g))\mid q^{e/2}+1=q^{k_\xi/2}+1=f'_q(o)$. Defining $h\in\C^2_{\rm conf}(g)$ such that $\lambda_{X-1}(h)=1$ and $\lambda_\xi(h)$ has order $f'_q(o)$, we see that $\ord_{\overline{G}}(\overline{h})=f'_q(o)$.

\vspace{3mm}

\emph{The case $G=\GU_\calU(q)$.} Here, as well, Equation~\eqref{eq:C^2_bil_un_cs} shows that the optimal exponent of $\C^2(\overline{g})$ is obtained when $P=Q$ and $0<q_{X-1}(g)\neq q_\chi(g)>0$ for all $\chi\in P\setminus\set{X-1}$, so that $T=\trivgrp$. Set
$$
f''_q(o)\coloneqq\begin{cases}
q^e+1 & \text{ if } f_{q^2}(o)=q^{2e}-1,\, e\text{ is odd and }o\mid q^e+1\\
f_{q^2}(o) & \text{otherwise}
\end{cases}.
$$
Equation~\eqref{eq:C^2_bil_un_cs} gives
\begin{equation}\label{eq:frmla_e_o_cprime_un_cs}
e_{\overline{G}}(o)=\exp(\C^2(\overline{g}))=\exp(\C^2_{\rm conf}(g)/Z)=\begin{cases}
1 & \text{if } o=1\\
f''_q(o) & \text{if } o>1
\end{cases}.
\end{equation}
Again $e_{\overline{G}}(1)=1$ is clear. If $o>1$, take $\mu\in\overline{k}^\times$ of order $o$ with minimal polynomial $\xi$. Then the argument proceeds as in the bilinear case. But the condition that $\xi$ is self-dual is here equivalent to $\mu$ being conjugate to $\mu^{-1}$ in $K_\xi$ by an automorphism $\alpha$ such that $\rest{\alpha}_k=\sigma$; $x\mapsto x^q$. This forces $e=k_\xi$ to be odd and $\mu^{q^{dk_\xi/2}+1}=1$, i.e., $o\mid q^e+1=q^{k_\xi}+1$.

\subsection{Proof of Theorem~\ref{thm:non_iso_main}}

Set $G_j\coloneqq X_{j\calU_j}(q_j)$, $Z_j\coloneqq\Z(G_j)$, and $\overline{G}_j\coloneqq G_j/Z_j$ ($j=1,2$). Let $p_j$ be the characteristic of the field $\finfield_{q_j}$ ($j=1,2$). Assume that $\overline{G}\cong\overline{G}_1\cong\overline{G}_2$. We start by showing that $p_1=p_2$.

\vspace{3mm}

\emph{Determining the characteristic $p$.} Choose $o$ large enough and coprime to $p_1$, $p_2$, $\card{Z_1}$, $\card{Z_2}$. Then from Equations~\eqref{eq:frmla_e_o_cprme_lin_cs}, \eqref{eq:frmla_e_o_cprime_bil_cs}, and \eqref{eq:frmla_e_o_cprime_un_cs} we see that $e_{\overline{G}}(o)$ is of the form $q_1^{e_1}\pm 1$ and $q_2^{e_2}\pm1$. If $e_{\overline{G}}(o)=q_1^{e_1}-1=q_2^{e_2}-1$ or $e_{\overline{G}}(o)=q_1^{e_1}+1=q_2^{e_2}+1$, we have $q_1^{e_1}=q_1^{e_2}$, so that $p_1=p_2$ by the uniqueness of the prime factorization. So, w.l.o.g., we have $e_{\overline{G}}(o)=q^{e_1}-1=q_2^{e_2}+1$ for infinitely many $o$, and so for infinitely many pairs $(e_1,e_2)\in\ints_+^2$. If $p_1\neq p_2$ we get a contradiction to Corollary~1.8 of \cite{evertse2019website}. Hence $p_1=p_2\eqqcolon p$.

\vspace{3mm}

\emph{Determining $q^d$.} We can now assume that $q_j=p^{e_j}$ ($j=1,2$). Choose $j\in\set{1,2}$ and set $X\coloneqq X_j$, $q\coloneqq q_j$, and $d\coloneqq d_j$. Consider the quantity $f\coloneqq\gcd\set{e_{\overline{G}}(o)}[o\in O]$, where
$$
O\coloneqq\set{o\in\ints_+}[2<o\text{ coprime to }p,\card{Z_1},\card{Z_2};\, e_{\overline{G}}(o)\equiv -1\text{ modulo }p^3]
$$
From Equations~\eqref{eq:frmla_e_o_cprme_lin_cs}, \eqref{eq:frmla_e_o_cprime_bil_cs}, and \eqref{eq:frmla_e_o_cprime_un_cs} it follows that for every element $o\in O$ the number $e_{\overline{G}}(o)$ is either of the form $q^{de}-1$ or $q^e+1$. But the second case is excluded by the condition that $e_{\overline{G}}(o)\equiv-1$ modulo $p^3$. Hence $q^d-1\mid e_{\overline{G}}(o)=q^{de}-1$ and so $q^d-1\mid f$.

For a prime $r$ set $t_r\coloneqq\frac{q^r-1}{q-1}$. Then for distinct primes $r$ and $s$ we have
\begin{align*}
\gcd\set{t_r,t_s}&=\gcd\lrset{\frac{q^r-1}{q-1},\frac{q^s-1}{q-1}}\\
&=\frac{1}{q-1}\gcd\set{q^r-1,q^s-1}=\frac{q^{\gcd\set{r,s}}-1}{q-1}=1.
\end{align*}
Hence the numbers $t_r$ ($r$ prime), being pairwise coprime, have arbitrarily large prime divisors. Take for $r>2$ a prime such that $t_r$ has a prime divisor $o>p,\card{Z_1},\card{Z_2},q^d-1$. Then $o$ is coprime to $p$, $\card{Z_1}$, and $\card{Z_2}$, so that by Equations~\eqref{eq:frmla_e_o_cprme_lin_cs}, \eqref{eq:frmla_e_o_cprime_bil_cs}, and \eqref{eq:frmla_e_o_cprime_un_cs} we have $e_{\overline{G}}(o)\mid f_{q^d}(o)\mid q^{dr}-1$, as $o\mid q^r-1\mid q^{dr}-1$. Hence the number $f_{q^d}(o)$ must be one of $q^{dr}-1$ or $q^d-1$, the latter being excluded by the condition $o>q^d-1$; so $f_{q^d}(o)=q^{dr}-1$. If $X=\GL$, $X=\Sp$, or $X=\GO$, since $r$ is odd and $d=1$, Equations~\eqref{eq:frmla_e_o_cprme_lin_cs} and \eqref{eq:frmla_e_o_cprime_bil_cs} show that we must have $e_{\overline{G}}(o)=f_{q^d}(o)=q^{dr}-1=q^r-1\equiv -1$ modulo $p^3$. Hence, in this case, $o\in O$. If $X=\GU$, it could be that $e_{\overline{G}}(o)=q^r+1$, when $o\mid q^r+1$. However, $\gcd\set{q^r+1,t_r}\mid\gcd\set{q^r+1,q^r-1}\mid 2$ and $t_r$ is always odd, so that $\gcd\set{q^r+1,t_r}=1$ and hence, as $o\mid t_r$, also $\gcd\set{q^r+1,o}=1$. This shows that here also $e_{\overline{G}}(o)=f_{q^d}(o)=q^{dr}-1=q^{2r}-1\equiv -1$ modulo $p^3$. Therefore again $o\in O$.

Applying this argument for two different primes $r$, say $r_1$ and $r_2$, which produces two different primes $o$, say $o_1$ and $o_2$, we get $f=\gcd\set{e_{\overline{G}}(o)}[o\in O]\mid\gcd\set{e_{\overline{G}}(o_1),e_{\overline{G}}(o_2)}=\gcd\set{q^{dr_1}-1,q^{dr_2}-1}=q^{\gcd\set{dr_1,dr_2}}-1=q^d-1$.

Altogether, we have shown that $f=q^d-1$. Plugging in $j=1,2$, we obtain $q_1^{d_1}-1=q_2^{d_2}-1$ implying that $q_1^{d_1}=q_2^{d_2}$.

\vspace{3mm}

Now we exclude all remaining possible isomorphisms but $\PSp_{\calU_1}(q)\cong\PGO_{\calU_2}(q)$.

\vspace{3mm}

\emph{Proof that $\PGL_{\calU_1}(q)\not\cong\PSp_{\calU_2}(q)$ and $\PGL_{\calU_1}(q)\not\cong\PGO_{\calU_2}(q)$.} 
Let $G_1=\GL_{\calU_1}(q)$ and $G_2=X_{\calU_2}(q)$, where $X=\Sp$ or $\GO$. Set
\begin{equation}\label{eq:def_o}
o\coloneqq\begin{cases}
\frac{q^2+1}{2} & \text{if }p>2\\
q^2+1 & \text{if }p=2
\end{cases}.
\end{equation}
Note that $o>2$ is coprime to $p$, $\card{Z_1}=q-1$, and $\card{Z_2}=\card{\set{\pm1}}$. Hence by Equation~\eqref{eq:frmla_e_o_cprime_bil_cs} we have $e_{\overline{G}}(o)=e_{\overline{G}_2}(o)=q^2+1$. Indeed, $o\mid q^2+1\mid q^4-1$.
But $o\nmid q^f-1$ for $f$ properly dividing $4$, since then $o\mid q^2-1$, but it is easy to see that $\gcd\set{o,q^2-1}=1$. This shows $f_q(o)=q^4-1$ and $e_{\overline{G}}(o)=e_{\overline{G}_2}(o)=f'_q(o)=q^2+1$. But then by $\overline{G}_1\cong\overline{G}_2$ we obtain $e_{\overline{G}}(o)=e_{\overline{G}_1}(o)=f_q(o)=q^4-1>q^2+1=e_{\overline{G}_2}(o)=e_{\overline{G}}(o)$, a contradiction.

\vspace{3mm}

\emph{Proof that $\PSp_{\calU_1}(q^2)\not\cong\PGU_{\calU_2}(q)$ and $\PGO_{\calU_1}(q^2)\not\cong\PGU_{\calU_2}(q)$.} Let $G_1=X_{\calU_1}(q^2)$, where $X=\Sp$ or $\GO$, and $G_2=\GU_{\calU_2}(q)$. Define $o$ as in Equation~\eqref{eq:def_o}. Note that $o>2$ is coprime to $p$, $\card{Z_1}=\card{\set{\pm1}}$, and $\card{Z_2}=q+1$. Then by Equation~\eqref{eq:frmla_e_o_cprime_bil_cs} we have $e_{\overline{G}}(o)=e_{\overline{G}_1}(o)=f'_{q^2}(o)=q^2+1$ (as above). But by Equation~\eqref{eq:frmla_e_o_cprime_un_cs} we obtain that $e_{\overline{G}}=e_{\overline{G}_2}(o)=f_q''(o)=q^4-1>q^2+1=e_{\overline{G}_1}(o)=e_{\overline{G}}(o)$, since $e=2$ is even, a contradiction.

\vspace{3mm}

\emph{Proof that $\PGL_{\calU_1}(q^2)\not\cong\PGU_{\calU_2}(q)$.} Let $G_1=\GL_{\calU_1}(q^2)$ and $G_2=\GU_{\calU_2}(q)$. Set
$$
o\coloneqq\begin{cases}
\frac{q^5+1}{5(q+1)} & \text{if }q\equiv -1\text{ modulo }5\\
\frac{q^5+1}{q+1} & \text{otherwise}
\end{cases}.
$$
Note that $o$ is coprime to $p$, $\card{Z_1}=q^2-1$, and $\card{Z_2}=q+1\mid q^2-1$. Indeed, $\gcd\set{o,q+1}\mid\gcd\set{\frac{q^5+1}{q+1},q+1}=\gcd\set{5,q+1}\mid 5$. But $5\nmid o$, so that $\gcd\set{o,q+1}=1$. Similarly, $\gcd\set{o,q-1}\mid\gcd\set{q^5+1,q-1}\mid\gcd\set{2,q-1}\mid 2$. But $o$ is always odd, so $\gcd\set{o,q-1}=1$. We have that $o\mid q^{10}-1$, so that from Equation~\eqref{eq:frmla_e_o_cprme_lin_cs} we obtain that $e_{\overline{G}}(o)=e_{\overline{G}_1}(o)=f_{q^2}(o)$ is either $q^{10}-1$ or $q^2-1$. But clearly $q^2-1<o$, so that we must have $f_{q^2}(o)=q^{10}-1$. But Equation~\eqref{eq:frmla_e_o_cprime_un_cs} gives that $e_{\overline{G}}(o)=e_{\overline{G}_2}(o)=f''_q(o)=q^5+1<q^{10}-1=f_q(o)=e_{\overline{G}_1}(o)$, a contradiction.

\begin{remark}
	If $q_i\to_\calU\infty$, then double centralizers of semisimple torsion elements are infinite groups.
\end{remark}

\begin{remark}
	If $q$ is even, then $\PSp_{\calU_1}(q)\cong\PGO_{\calU_2}(q)$ is possible due to the isomorphism $\Sp_{2m}(q)\cong\GO_{2m+1}(q)$.
	Also it seems hard to distinguish a group $\PSp_{\calU_1}(q)$ from a group $\PGO_{\calU_2}(q)$ for $q$ odd.
\end{remark}

\section*{Acknowledgments}
The author wants to thank Andreas Thom for interesting discussions about the topic. The content of this article is part of the PhD project \cite{schneider2019phd} of the author. This research was supported by ERC Consolidator Grant No.\ 681207.

\end{document}